\documentclass[11pt]{amsart}
\usepackage{float}
\usepackage[usenames]{color}
\usepackage{graphicx}
\usepackage{amssymb}
\usepackage{amscd}
\theoremstyle{plain}
\newtheorem{thm}{Theorem}[section]
\newtheorem{lemma}[thm]{Lemma}
\newtheorem{cor}[thm]{Corollary}

\newtheorem{prop}[thm]{Proposition}

\theoremstyle{definition}
\newtheorem{definition}[thm]{Definition}
\newtheorem{remark}[thm]{Remark}
\newtheorem{defi}[thm]{Definition}

\newtheorem{rmk}[thm]{Remark}
\newtheorem{example}[thm]{Example}

\def\dim{\mathop{\hbox {dim}}\nolimits}

\def\det{\mathop{\hbox {det}}\nolimits}

\def\im{\mathop{\hbox {Im}}\nolimits}

\def\ker{\mathop{\hbox{Ker}}\nolimits}

\def\aq{A_{\mathfrak{q}}}

\newcommand{\fra}{\mathfrak{a}}

\newcommand{\frg}{\mathfrak{g}}
\newcommand{\frh}{\mathfrak{h}}

\newcommand{\frk}{\mathfrak{k}}
\newcommand{\frl}{\mathfrak{l}}

\newcommand{\fro}{\mathfrak{o}}
\newcommand{\frp}{\mathfrak{p}}
\newcommand{\frq}{\mathfrak{q}}

\newcommand{\frs}{\mathfrak{s}}
\newcommand{\frt}{\mathfrak{t}}
\newcommand{\fru}{\mathfrak{u}}

\newcommand{\bbC}{\mathbb{C}}

\newcommand{\bbR}{\mathbb{R}}

\newcommand{\bbZ}{\mathbb{Z}}

\newcommand{\caA}{\mathcal{A}}
\newcommand{\caB}{\mathcal{B}}
\newcommand{\caC}{\mathcal{C}}

\newcommand{\caK}{\mathcal{K}}
\newcommand{\caL}{\mathcal{L}}

\newcommand{\caR}{\mathcal{R}}

\newcommand{\be}{\begin {equation}}
\newcommand{\ee}{\end {equation}}
\newcommand{\bp}{\begin {proof}}
\newcommand{\ep}{\end {proof}}

\begin{document}

\title[Dirac index of some unitary representations of $Sp(2n, \bbR)$ and $SO^*(2n)$]{Dirac index of some unitary representations of $Sp(2n, \bbR)$ and $SO^*(2n)$}
\author{Chao-Ping Dong}
\author{Kayue Daniel Wong}

\address[Dong]{School of Mathematical Sciences, Soochow University, Suzhou 215006,
P.~R.~China}
\email{chaopindong@163.com}

\address[Wong]{School of Science and Engineering, The Chinese University of Hong Kong, Shenzhen,
Guangdong 518172, P. R. China}
\email{kayue.wong@gmail.com}

\abstract{Let $G$ be $Sp(2n, \bbR)$ or $SO^*(2n)$. We compute the Dirac index of a large class of unitary representations considered by Vogan in Section 8 of \cite{Vog84}, which include all weakly fair $\aq(\lambda)$ modules and (weakly) unipotent representations of $G$ as two extreme cases. We conjecture that these representations exhaust all unitary representations of $G$ with nonzero Dirac cohomology. In general, for certain irreducible unitary module of an equal rank group, we clarify the link between the possible cancellations in its Dirac index, and the parities of its spin-lowest $K$-types.}
 \endabstract



\maketitle
\section{Introduction}

Let $G$ be a connected linear Lie group with Cartan involution $\theta$. Denote by $\widehat{G}$ the unitary dual of $G$, which stands for the set of all the equivalence classes of irreducible unitary representations of $G$.
Classification of $\widehat{G}$ is a fundamental problem in representation theory of Lie groups. An algorithm answer has been given by Adams, van Leeuwen, Trapa and Vogan \cite{ALTV}. Including unitarity detecting, the corresponding software implementation  \texttt{atlas} \cite{At} now computes many aspects of problems in Lie theory.

To get a better understanding of $\widehat{G}$, Vogan introduced Dirac cohomology in 1997 \cite{Vog97}. Let us build up necessary notation to recall this notion. Assume that $K:=G^\theta$ is a maximal compact subgroup of $G$.  Choose a maximal torus $T$ of $K$. Write the Lie algebra of $G$, $K$ and $T$ as $\frg_0$,  $\frk_0$ and $\frt_{0}$, respectively. Let
 $$
 \frg_0=\frk_0\oplus\frp_0
 $$
 be the Cartan decomposition on the Lie algebra level.
 Let $\fra_{0}$ be the centralizer of $\frt_{0}$ in $\frp_0$, and put $A=\exp(\fra_{0})$. Then $H=TA$ is the unique maximally compact $\theta$-stable Cartan subgroup of $G$, and
 $\frh_{0}:=\frt_{0}\oplus\fra_{0}$ is the fundamental Cartan subalgebra of $\frg_0$. We will write $\frg=\frg_0\otimes_{\bbR} \bbC$ and so on.

Fix a non-degenerate invariant symmetric bilinear form $B$ on $\frg$. We may also write $B$ as $\langle \cdot, \cdot \rangle$. Then $\frk$ and $\frp$ are orthogonal to each other under $B$. Fix an orthonormal basis $\{Z_1, \dots, Z_m\}$ of $\frp_0$ with respect to the inner product on $\frp_0$ induced by $B$. Let $U(\frg)$ be the universal enveloping algebra, and let $C(\frp)$ be the Clifford algebra. As introduced by Parthasarathy \cite{Pa},
the \emph{Dirac operator} is defined as
\begin{equation}\label{Dirac-operator}
D:=\sum_{i=1}^m Z_i\otimes Z_i\in U(\frg)\otimes C(\frp),
\end{equation}
which is independent of the choice of the orthonormal basis $\{Z_i\}_{i=1}^m$.
Let ${\rm Ad}: K\to SO(\frp_0)$ be the adjoint map, and $p: {\rm Spin}(\frp_0) \to SO(\frp_0)$ be the universal covering map. Then
$$
\widetilde{K}:=\{ (k, s) \in K \times {\rm Spin}(\frp_0)\mid {\rm Ad}(k)=p(s) \}
$$
is the spin double cover of $K$. Let $\pi$ be any $(\frg, K)$ module.
The Dirac operator $D$ acts on $\pi\otimes {\rm Spin}_G$, where ${\rm Spin}_G$ is a spin module for the Clifford algebra $C(\frp)$. The \emph{Dirac cohomology} is defined \cite{Vog97} as the following $\widetilde{K}$-module:
\begin{equation}\label{Dirac-cohomology}
H_D(\pi):=\ker D/(\ker D\cap\im D).
\end{equation}

 Fix a positive root system $\Delta^+(\frk, \frt)$ once for all, and denote the half sum of roots in $\Delta^+(\frk, \frt)$ by $\rho_c$. We will use $E_{\mu}$ to denote the $\frk$-type (that is, an irreducible representation of $\frk$) with highest weight $\mu$. Abuse the notation a bit, $E_{\mu}$ will also stand for the $K$-type as well as the $\widetilde{K}$-type with highest weight $\mu$.

The following Vogan conjecture, proved by Huang and Pand\v zi\'c \cite{HP}, is foundational for computing $H_D(\pi)$.

\begin{thm}\emph{(\textbf{Huang-Pand\v zi\'c} \cite{HP})}\label{thm-HP}
Let $\pi$ be any irreducible $(\frg, K)$ module with infinitesimal character $\Lambda\in\frh^*$. Assume that $H_D(\pi)$ is non-zero, and that $E_{\gamma}$ is contained in $H_D(\pi)$. Then $\Lambda$ is conjugate to $\gamma+\rho_c$ by some element in the Weyl group $W(\frg, \frh)$.
\end{thm}

Let $\widehat{G}^d$ be the \emph{Dirac series} of $G$. That is, the members of $\widehat{G}$ with non-zero Dirac cohomology. Classification of Dirac series is a smaller project than the classification of the unitary dual. Yet it is still worthy of pursuing  since Dirac series contains many interesting unitary representations such as the discrete series \cite{HP},  certain $A_\frq(\lambda)$ modules \cite{HKP} and beyond. Moreover, due to the research announcement of Barbasch and Pand\v zi\'c \cite{BP19}, Dirac series should have applications in the theory of automorphic forms.

Recently, Dirac series has been classified for complex classical Lie groups \cite{BDW,DW1,DW2} and $GL(n, \bbR)$ \cite{DW3}.
For other classical groups such as $U(p, q)$, $Sp(2n, \bbR)$ and $SO^*(2n)$ whose unitary dual is unknown, it is still very hard to achieve the complete classification of their Dirac series. In this case, the \emph{Dirac index} can offer some help: it is much easier to compute and whenever it is non-zero, the Dirac cohomology must be non-zero.

Let us build up a bit more notation for introducing Dirac index. Unless stated otherwise, we further assume that $G$ is equal rank henceforth. Then $\frh=\frt$ and $\fra=0$. Put $\frt_{\bbR}=i\frt_0$ and $\frt_{\bbR}^*=i\frt_0^*$. Let $\caC\subseteq \frt_{\bbR}^*$ be the dominant Weyl chamber for $\Delta^+(\frk, \frt)$. Choose a positive root system
$$
\Delta^+(\frg, \frt)=\Delta^+(\frk, \frt) \cup \Delta^+(\frp, \frt).
$$
Let $\caC_{\frg}\subseteq \frt_{\bbR}^*$ be the dominant Weyl chamber for $\Delta^+(\frg, \frt)$.
Let
\begin{equation}
W(\frg, \frt)^1=\{w\in W(\frg, \frt) \mid w \caC_{\frg} \subseteq \caC\}.
\end{equation}
The set $W(\frg, \frt)^1$ has cardinality $s:=|W(\frg, \frt)|/|W(\frk, \frt)|$. Any positive root system $(\Delta^+)^\prime(\frg, \frt)$ of $\Delta(\frg, \frt)$  containing $\Delta^+(\frk, \frt)$ has the form
$$
(\Delta^+)^\prime(\frg, \frt)=\Delta^+(\frk, \frt) \cup (\Delta^+)^\prime(\frp, \frt)
$$
with $(\Delta^+)^\prime(\frp, \frt)=w\Delta^+(\frp, \frt)$ for some $w\in W(\frg, \frt)^1$.
Denote the half sum of roots in $\Delta^+(\frg, \frt)$ (resp., $\Delta^+(\frp, \frt))$ by $\rho$ (resp., $\rho_n$). The notation $\rho^\prime$ and $\rho_n^\prime$ will be interpreted similarly. Then
$$
\rho=\rho_c + \rho_n, \quad  \rho^\prime=\rho_c + \rho_n^\prime.
$$
Let
$$
\frp^+=\sum_{\alpha\in \Delta^+(\frp, \frt)} \frg_{\alpha}, \quad
\frp^-=\sum_{\alpha\in \Delta^+(\frp, \frt)} \frg_{-\alpha}.
$$
Then $\frp=\frp^+ \oplus \frp^-$ and
$$
{\rm Spin}_{G}\cong \bigwedge \frp^+ \otimes \bbC_{-\rho_n}.
$$
Any weight in ${\rm Spin}_{G}$ has the form $-\rho_n + \langle\Phi\rangle$, where $\Phi$ is a subset of $\Delta^+(\frp, \frt)$ and $\langle\Phi \rangle$ stands for the sum of the roots in $\Phi$.
Now put
\begin{equation}\label{spin-module-even-odd}
{\rm Spin}_{G}^+ = \bigwedge^{\rm even} \frp^+ \otimes \bbC_{-\rho_n}, \quad {\rm Spin}_{G}^- = \bigwedge^{\rm odd} \frp^+ \otimes \bbC_{-\rho_n}.
\end{equation}
The Dirac operator $D$ interchanges $\pi\otimes {\rm Spin}^+_{G}$ and $\pi\otimes {\rm Spin}^-_{G}$. Thus the Dirac cohomology $H_D(\pi)$ breaks up into the even part $H_D^+(\pi)$ and the odd part $H_D^-(\pi)$. The \emph{Dirac index} of $\pi$ is defined as
the virtual $\widetilde{K}$-module
\begin{equation}\label{Dirac-index}
{\rm DI}(\pi):=H_D^+(\pi)-H_D^-(\pi).
\end{equation}
By Remark 3.8 of \cite{MPVZ}, if $(\Delta^+)^\prime(\frg, \frt)$ is chosen instead of $\Delta^+(\frg, \frt)$, one has
$$
{\rm DI}^\prime(\pi)=(-1)^{\# ((\Delta^+)^{\prime}(\frp, \frt) \setminus \Delta^+(\frp, \frt)) }{\rm DI}(\pi).
$$
Therefore, the Dirac index is well-defined up to a sign. Moreover, by Proposition 3.12 of \cite{MPV},
$$
{\rm DI}(\pi)=\pi\otimes {\rm Spin}^+_{G}- \pi\otimes {\rm Spin}^-_{G}.
$$
It turns out that the Dirac index preserves short exact sequences and has nice behavior with respect to coherent continuation \cite{MPV,MPVZ}. This idea is pursued in \cite{DW4} to compute the Dirac index of all weakly fair $\aq(\lambda)$-modules for $G = U(p,q)$.

In this paper, we study a larger class of unitary representations
constructed in \cite[Section 8]{Vog84} (see Theorem \ref{thm-Vog84} below).
This includes all weakly fair $\aq(\lambda)$-modules and unipotent representations as two extreme cases.
We will compute the Dirac index for all such representations for $G = Sp(2n, \bbR)$ and $SO^*(2n)$,
and conjecture that composition factors of these representations should exhaust the Dirac series of $G$.


\medskip
The paper is organized as follows: In Section \ref{sec-pre}, we provide the preliminaries for Dirac index and cohomological induction. We also describe the class of representations that we are interested in. In Sections \ref{sec-sp2nr} and \ref{sec-sostar2n}, we study the Dirac index of all unipotent representations of $G = Sp(2n,\mathbb{R})$ and $SO^*(2n)$. In Section \ref{sec-general}, we compute the Dirac index of all unitary representations of $G$ covered in Theorem \ref{thm-Vog84}. Finally, in Section \ref{sec-parity}, we reveal the relation between the possible cancellations in $H_D(\pi)$, and the parities of the spin-lowest $K$-types of $\pi$, where $\pi$ is irreducible unitary. See Theorem \ref{thm-parity-cancellation}.

\section{Preliminaries}\label{sec-pre}

We continue with the notation in the introduction. In particular, $G$ is equal rank.

\subsection{Dirac index}\label{sec-DI}
In this section, we choose \emph{the} Vogan diagram for $\frg_0$ as Appendix C of Knapp \cite{Kn}. Then we have actually chosen a
$$
\Delta^+(\frg, \frt)=\Delta^+(\frk, \frt)\cup \Delta^+(\frp, \frt).
$$
The Vogan diagram for $\frg_0$ has a unique black dot, which stands for a simple root $\gamma$. We denote the fundamental weight corresponding to $\gamma$ by $\widetilde{\zeta}$. Put
\begin{equation}\label{zeta}
\zeta=\frac{2}{\|\gamma\|^2} \widetilde{\zeta}.
\end{equation}

The following result should be well-known, and it can be obtained by going through the classification of real simple Lie algebras.

\begin{lemma} Let $\beta$ be the highest root in $\Delta^+(\frg, \frt)$. Then $G/K$ is  Hermitian symmetric if and only if the unique black dot simple root has coefficient $1$ in $\beta$. Otherwise, the unique black dot simple root must have coefficient $2$ in $\beta$.
As a consequence, we always have that for any $\alpha\in\Delta(\frg, \frt)$,
\begin{equation}\label{zeta-k-p}
\langle \alpha, \zeta\rangle \mbox{ is even} \Leftrightarrow  \alpha\in \Delta(\frk, \frt), \quad \langle \alpha, \zeta\rangle \mbox{ is odd} \Leftrightarrow  \alpha\in \Delta(\frp, \frt).
\end{equation}
\end{lemma}

\begin{lemma}\label{lemma-Phiw}
For any $w\in W(\frg, \frt)^1$, the set $\Phi_w:=w\Delta^-(\frg, \frt)\cap \Delta^+(\frg, \frt)$ must be contained in $\Delta^+(\frp, \frt)$.
\end{lemma}
\begin{proof}
Suppose there exists a $w\in W(\frg, \frt)^1$ such that $\Phi_w$ is not contained in $\Delta^+(\frp, \frt)$. Then we can find a root $\alpha\in \Delta^+(\frk, \frt)$ such that $\alpha\in \Phi_w$. By the definition of $\Phi_w$, we can further find a root $\beta\in\Delta^+(\frg, \frt)$ such that $\alpha=-w\beta$. That is, $-\alpha=w\beta$. Now,
$$
\langle -\alpha, w\rho\rangle=\langle w\beta, w\rho\rangle=\langle \beta, \rho\rangle>0.
$$
Therefore, $\langle w\rho, \alpha\rangle<0$. This contradicts to the assumption that $w\in W(\frg, \frt)^1$ since $w\rho$ should be dominant for $\Delta^+(\frk, \frt)$.
\end{proof}

The first named author learned the following result from Pand\v zi\'c. It should be well-known to the experts.
\begin{lemma}\label{lemma-even-odd-spinG}
 We have the following decompositions of the spin module into $\frk$-types:
\begin{equation}\label{spin-Weyl}
{\rm Spin}^+_G =\bigoplus_{\scriptsize l(w)\, \mbox{even}} E_{w\rho-\rho_c}, \quad
{\rm Spin}^-_G =\bigoplus_{\scriptsize l(w)\, \mbox{odd}} E_{w\rho-\rho_c},
\end{equation}
where $w$ runs over the set $W(\frg, \frt)^1$.
\end{lemma}
\begin{proof}
As we know,
$$
{\rm Spin}_G=\bigwedge\frp^+\bigotimes \bbC_{-\rho_n}\cong \bigoplus_{w\in W(\frg, \frt)^1} V_{w\rho-\rho_c}.
$$
It remains to separate the even part and the odd part of the spin module.
Take an arbitrary $w\in W(\frg, \frt)^1$. Then there exists a subset $\Phi$ of $\Delta^+(\frp, \frt)$ such that
$$
w\rho - \rho_c=\rho_n - \langle\Phi\rangle,
$$
where $\langle\Phi\rangle$ stands for the sum of roots in $\Phi$. One deduces from the above equality that
$$
\langle\Phi_w\rangle=\langle\Phi\rangle.
$$
Taking inner products of the two sides of the above equality with $\zeta$, we have that
\begin{equation}\label{spin-w}
l(w)\equiv |\Phi| \quad ({\rm mod} \ 2)
\end{equation}
by recalling \eqref{zeta-k-p} and Lemma \ref{lemma-Phiw}.
Note that any weight of $E_{w\rho- \rho_c}$ can be obtained by subtracting some roots of $\Delta^+(\frk, \frt)$ from $w\rho-\rho_c$. Now the desired result follows from \eqref{zeta-k-p}.
\end{proof}

Let $\pi$ be an irreducible unitary $(\frg, K)$ module. Assume that $H_D(\pi)$ is non-zero. Then any $\widetilde{K}$-type $E_{\gamma}$ of $H_D(\pi)$ lives in either $\pi\otimes {\rm Spin}_G^+$ or $\pi\otimes {\rm Spin}_G^-$. We assign a sign to $E_{\gamma}$ as follows: the sign is $+1$ in the first case, it is $-1$ otherwise.
We will refer to this sign as the \emph{sign} of $E_{\gamma}$ in $H_D(\pi)$.

\subsection{Cohomologically induced modules}\label{sec-coho}
Firstly, let us fix an element $H\in \frt_{\bbR}$ which is dominant for $\Delta^+(\frk, \frt)$, and define the $\theta$-stable parabolic subalgebra
\begin{equation}\label{def-theta-stable-parabolic}
\frq=\frl\oplus \fru
\end{equation}
as the nonnegative eigenspaces of ${\rm ad}(H)$. The Levi subalgebra $\frl$ of $\frq$ is the zero eigenspace of ${\rm ad}(H)$, while the nilradical $\fru$ of $\frq$ is the sum of positive eigenspaces of ${\rm ad}(H)$. If we denote  by $\overline{\fru}$ the sum of negative eigenspaces of ${\rm ad}(H)$, then
$$
\frg=\overline{\fru} \oplus \frl \oplus \fru.
$$
Let $L$ be the normalizer of $\frq$ in $G$. Then $L\cap K$ is a maximal compact subgroup of $L$. Let $\mathfrak{z}$ be the center of $\frl$.

We choose a positive root system $(\Delta^+)^\prime(\frg, \frt)$ containing $\Delta(\fru, \frt)$ and $\Delta^+(\frk, \frt)$. Set
$$
\Delta^+(\frl, \frt)=\Delta(\frl, \frt)\cap (\Delta^+)^\prime(\frg, \frt).
$$
Denote by $\rho^L$ (resp., $\rho_c^L$) the half sum of positive roots in $\Delta^+(\frl, \frt)$ (resp., $\Delta^+(\frl\cap\frk, \frt)$). Let $\rho_n^L=\rho^L-\rho_c^L$. Denote by $\rho(\fru)$ (resp., $\rho(\fru\cap\frp)$, $\rho(\fru\cap\frk)$) the half sum of roots in $\Delta(\fru, \frt)$ (resp., $\Delta(\fru\cap \frp, \frt)$, $\Delta(\fru\cap\frk, \frt)$).
The following relations hold
\begin{equation}\label{half-sum-relations}
\rho^{(j)}=\rho^L+\rho(\fru), \quad \rho_c=\rho_c^L+\rho(\fru\cap\frk), \quad \rho_n^{(j)}=\rho_n^L+\rho(\fru\cap\frk).
\end{equation}

The cohomological induction functors $\caL_j(\cdot)$ and $\caR^j(\cdot)$ lift an admissible $(\frl, L\cap K)$ module $Z$ to $(\frg, K)$ modules, and the most interesting case happens at the middle degree $S:=\dim (\fru\cap\frk)$. Assume that $Z$ has real infinitesimal character $\Lambda_Z\in\frt_{\bbR}^*$.
After \cite{KV}, we say that $Z$ is in the \emph{good range} (relative to $\frq$ and $\frg$) if
\begin{equation}\label{good-range}
\langle \Lambda_Z+\rho(\fru), \alpha \rangle > 0, \quad \forall \alpha\in \Delta(\fru, \frt).
\end{equation}
We say that $Z$ is in the \emph{weakly good range} if
\begin{equation}\label{weakly-good-range}
\langle \Lambda_Z+\rho(\fru), \alpha \rangle \geq 0, \quad \forall \alpha\in \Delta(\fru, \frt).
\end{equation}
Moreover, $Z$ is said to be in the \emph{fair range} if
\begin{equation}\label{fair-range}
\langle \Lambda_Z+\rho(\fru), \alpha|_{\mathfrak{z}} \rangle > 0, \quad \forall \alpha\in \Delta(\fru, \frt).
\end{equation}
We say that $Z$ is in the \emph{weakly fair range} if
\begin{equation}\label{weakly-fair-range}
\langle \Lambda_Z+\rho(\fru), \alpha|_{\mathfrak{z}} \rangle \geq 0, \quad \forall \alpha\in \Delta(\fru, \frt).
\end{equation}

When the inducing module $Z$ is a one-dimensional unitary character $\bbC_{\lambda}$, we denote the corresponding $(\frg, K)$-module $\caL_S(Z)$ by $A_{\frq}(\lambda)$, which has infinitesimal character $\lambda+\rho^\prime$.
Good range $A_\frq(\lambda)$ modules must be non-zero, irreducible, and unitary \cite{KV}.
They play an important role in the unitary dual. Indeed, as shown by Salamanca-Riba \cite{Sa}, any irreducible unitary $(\frg, K)$-module with a real, integral, and strongly regular infinitesimal character $\Lambda$ must be isomorphic to an $A_{\frq}(\lambda)$ module in the good range.  Here  $\Lambda$ being \emph{strongly regular} means that
$$
\langle \Lambda-\rho^\prime, \alpha \rangle \geq 0, \quad \forall \alpha\in (\Delta^+)^\prime(\frg, \frt).
$$

Note that the module $\aq(\lambda)$ is weakly fair if
\begin{equation}\label{Aqlambda-weakly-fair}
\langle \lambda+\rho(\fru), \alpha \rangle \geq 0, \quad \forall \alpha\in \Delta(\fru, \frt).
\end{equation}
A weakly fair $A_{\frq}(\lambda)$ module can be zero or reducible. However, whenever it is non-zero, it must be unitary \cite{KV}. More importantly, many singular unitary representations can be realized as (a composition factor of) a weakly fair $A_{\frq}(\lambda)$ module.

Based on \cite{MPV,MPVZ}, the following result computes the Dirac index of weakly fair $A_{\frq}(\lambda)$.

\begin{thm}\label{thm-DI-Aqlambda} \emph{(Theorem 4.3 of \cite{DW4})}
The Dirac index of weakly fair $\aq(\lambda)$ is equal to
$$
{\rm DI}(A_{\frq}(\lambda)) = \sum_{w \in W(\mathfrak{l},\mathfrak{t})^1} {\rm det}(w) \widetilde{E}_{w(\lambda + \rho)},$$
where
\begin{equation}\label{E-tilde-mu}
\widetilde{E}_{\mu} = \begin{cases}
0 & \text{if}\ \mu\ \text{is}\ \Delta(\mathfrak{k},\mathfrak{t})\text{-singular} \\
\mathrm{det}(w) E_{w\mu - \rho_c} & \text{if } \exists w \in W(\mathfrak{k},\mathfrak{t})\ \text{s.t. } w\mu\ \text{is dominant regular for } \Delta^+(\mathfrak{k},\mathfrak{t})
\end{cases}.
\end{equation}
\end{thm}

\subsection{A larger class of unitary modules}
As stated in the previous section, all weakly fair $A_{\frq}(\lambda)$-modules are unitarty. Indeed,
\cite{Vog84} obtained a unitarity theorem for a larger class of representations.

\begin{thm}[\cite{Vog84} Proposition 8.17] \label{thm-Vog84}
Let $G$ be a reductive Lie group, and $\mathfrak{q} = \mathfrak{l} + \mathfrak{u}$ be a $\theta$-stable parabolic subalgebra of $\mathfrak{g}$ such that
$$\mathfrak{l} \cong \mathfrak{g}_{A_l} + \dots + \mathfrak{g}_{A_1} + \mathfrak{g}_m',$$
where each $\mathfrak{g}_{A_t}$ is of Type $A$, and $\mathfrak{g}_m'$ is not of Type $A$ with rank $m$
(here we allow $l = 0$, i.e., there are no Type $A$ factors and $\mathfrak{l} = \mathfrak{g}$; or $m = 0$, i.e, all
factors are of Type $A$).

Suppose $Z$ is an $(\mathfrak{l}, L \cap K)$ module given by the tensor product of
unitary characters of Type $A$ and a \emph{weakly unipotent representation} $\pi_u$ of $G_m'$ \cite[Definition 8.16]{Vog84}, such that the infinitesimal character $\Lambda_Z$ of $Z$ is in the weakly fair range. Then $\mathcal{L}_S(Z)$ is an unitary $(\mathfrak{g},K)$ module.
\end{thm}

If the weakly unipotent representation in Theorem \ref{thm-Vog84} is taken to be the trivial representation, or $\mathfrak{g}_m'$ does not exist in $\mathfrak{l}$, then we are in the
setting of $\aq(\lambda)$ modules in the previous subsection. On the other extreme, if $\mathfrak{l} = \mathfrak{g} = \mathfrak{g}_m'$, we obtain all (weakly) unipotent representations.

%

We are interested in representations in Theorem \ref{thm-Vog84} with non-zero Dirac cohomology. A necessary condition for such a module to have non-zero Dirac cohomology
is that the weakly unipotent representation $\pi_u$ of $\mathfrak{g}_m'$ has infinitesimal character
satisfying Theorem \ref{thm-HP}. For $G = U(p,q)$ and $Sp(2n,\mathbb{R})$, the classification of all such $\pi_u$ along with their Dirac cohomologies are given in \cite{BP15}.

From now on, we focus on $G = Sp(2n,\mathbb{R})$ and $SO^*(2n)$.  We begin by computing $\mathrm{DI}(\pi_u)$ for all unipotent representations $\pi_u$ of $G$ whose infinitesimal characters satisfy Theorem \ref{thm-HP}.
From there, one can use the results in \cite{DW4} to compute $\mathrm{DI}(\mathcal{L}_S(Z))$
in Theorem \ref{thm-Vog84} in the special case when $Z$ consists of trivial characters
for each Type $A$ factor, and $\pi_u$ for the $G_m'$ factor (Theorem \ref{thm-DH}).

As for the general case of weakly fair $\mathcal{L}_S(Z')$ such that $Z'$ consists of some Type $A$ characters and the same $\pi_u$ for the $G_m'$ factor, then it can be placed inside a \emph{coherent family} containing the representation $\mathcal{L}_S(Z)$ given in the previous paragraph (c.f. \cite[Theorem 7.2.23]{Vog81}). Consequently, one can get $\mathrm{DI}(\mathcal{L}_S(Z'))$ from $\mathrm{DI}(\mathcal{L}_S(Z))$ by the translation principle (Proposition \ref{prop-translation}). The explicit formulas for $\mathrm{DI}(\mathcal{L}_S(Z'))$ are given in Corollary \ref{cor-mainsp} and Theorem \ref{thm-mainso}.

\section{Dirac index of unipotent representations of $Sp(2n, \bbR)$} \label{sec-sp2nr}

Let $G$ be $Sp(2n, \bbR)$. We fix a Vogan diagram for $\frs\frp(2n, \bbR)$ as in Fig.~\ref{Fig-Sp2n-Vogan}, where the simple roots are $e_1-e_2$, $e_2-e_3$, $\dots$, $e_{n-2}-e_{n-1}$, $e_{n-1}-e_n$ and $2e_n$ when enumerated  from left to right, with $2e_n$ being dotted. Therefore, $\zeta=(\frac{1}{2}, \dots, \frac{1}{2})$.

\begin{figure}[H]
\centering
\scalebox{0.80}{\includegraphics{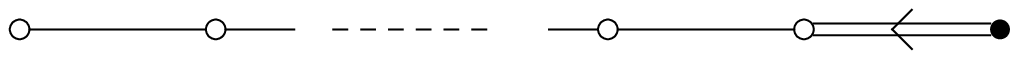}}
\caption{The Vogan diagram for $\frs\frp(2n, \bbR)$}
\label{Fig-Sp2n-Vogan}
\end{figure}

By fixing this Vogan diagram, we have actually fixed
$$
\Delta^+(\frk, \frt)=\{e_i-e_j \mid 1\leq i<j \leq n\},
$$
and chosen
$$
\Delta^+(\frp, \frt)=\{e_i+e_j \mid 1\leq i<j \leq n\}\cup \{2e_i\mid 1\leq i\leq n\}.
$$
Then $\rho_c=(\frac{n-1}{2}, \frac{n-3}{2}, \dots, -\frac{n-1}{2})$, $\rho_n=(\frac{n+1}{2},  \dots, \frac{n+1}{2})$, and $\rho=(n, n-1, \dots, 1)$.

\medskip
In \cite{BP15}, Barbasch and Pand\v zi\'c gave a construction of all
unipotent representations $\pi_u$ of $Sp(2n, \bbR)$ whose infinitesimal characters satisfy
Theorem \ref{thm-HP}.
Therefore, it includes all possible unipotent representations of $G$
with non-zero Dirac cohomology. In the next couple of subsections, we will study
the Dirac index of these representations.

\subsection{Dirac index of $\pi_u = X(r,s;\epsilon,\eta)$}\label{sec-DI-Xpqepsiloneta}

Let $k$ be a positive integer such that $r+s=2k\leq n$, where $r$ and $s$ are non-negative integers. Let $\epsilon, \eta$ be $0$ or $1$. Set $\epsilon=0$ if $r=0$, and set $\eta=0$ if $s=0$. The irreducible unipotent representation $X(r,s; \epsilon, \eta)$ is constructed from the character $\bbC_{\epsilon, \eta}$ of $O(r, s)$ via theta lifting. Its $K$-types are as follows:
\begin{equation} \label{sp2nktypes}
(\frac{r-s}{2}, \dots, \frac{r-s}{2})+(2a_1+\epsilon, \dots, 2a_r+\epsilon, 0, \dots, 0, -2b_s-\eta, \dots, -2b_1-\eta),
\end{equation}
where $a_1\geq \cdots \geq a_r\geq 0$ and $b_1\geq \cdots \geq b_s\geq 0$ are integers.
The infinitesimal character of $X(r, s; \epsilon, \eta)$ is
$$
\Lambda=\Lambda_k=(n-k, n-k-1, \dots, k+1, k, k-1, \dots, -k+1).
$$
As in \cite{BP15}, we call the last $2k$ coordinates $k, k-1, \dots, -k+1$ the \emph{core} of $\Lambda$, and call the first $n-2k$ coordinates $n-k, n-k-1, \dots, k+1$ the \emph{tail} of $\Lambda$. For $w\in W(\frg, \frt)$, $w\Lambda -\rho_c$ is the highest weight of a $\widetilde{K}$-type if and only if the entries of $w \Lambda$ are strictly decreasing. In such a case, we must have that either
\begin{equation}\label{special}
w\Lambda=(i_1, \dots, i_u, k, k-1, \dots, -k+1, -j_v, \dots, -j_1),
\end{equation}
or
\begin{equation}\label{non-special}
w\Lambda=(i_1, \dots, i_u,  k-1, \dots, -k+1, -k, -j_v, \dots, -j_1).
\end{equation}
Here $u$ and $v$ are two non-negative integers such that $u+v=n-2k$. Moreover, $i_1> \cdots> i_u$, $j_1> \cdots> j_v$ are such that
$$
\{i_1, \dots, i_u, j_1, \dots, j_v \}
=\{k+1, k+2, \dots, n-k\}.
$$
Let $\tau:=w\Lambda -\rho_c$ be a $\widetilde{K}$-type. We call $\tau$ \emph{special} if \eqref{special} holds, and call $\tau$ \emph{non-special} if \eqref{non-special} holds. Let $w_0^K$ be the longest element of $W(\frk, \frt)$, which is isomorphic to $S_n$. Then one sees that $\tau$ is special if and only if $-w_0^K\tau$ is non-special.

Let $\tau:=w\Lambda -\rho_c$ be a $\widetilde{K}$-type with $w\Lambda$ given by \eqref{special}. In particular, $\tau$ is special.  Let $i_t^\prime=i_t - k$ and $j_t^\prime=j_t -k$. Then
$$
\{i'_1, \dots, i'_u, j'_1, \dots, j'_v \}
=\{1, 2, \dots, n-2k\}.
$$
Put
\begin{equation}\label{M-tau-special}
M(\tau):=j_1^\prime + \cdots + j_v^\prime + r(\epsilon + v) +\frac{r(r-1)}{2},
\end{equation}
and
\begin{equation}\label{N-tau-special}
N(\tau):=j_1^\prime + \cdots + j_v^\prime + s(\eta + v) +\frac{s(s-1)}{2}+\frac{n(n+1)}{2}.
\end{equation}

When $\tau$ is special, it follows from the proof of \cite[Theorem 3.7]{BP15} that the sign of the $\widetilde{K}$-type $E_{\tau}$ in $H_D(X(r, s; \epsilon, \eta))$ (if it does occur) is $(-1)^{M(\tau)}$. Indeed, in the setting of \cite{BP15}, we look at all possibilities of $\sigma \in W(\mathfrak{g},\mathfrak{t})^1$ such that
\begin{equation}\label{BP-sigma}
\sigma \rho=(x_1, \dots, x_s, i_1^\prime, \dots, i_u^\prime, -j_v^\prime, \dots, -j_1^\prime, -y_r, \dots, -y_1),
\end{equation}
where
\begin{align*}
&x_1=\eta+u+2\widetilde{b_1}+s,\quad \dots,\quad x_s=\eta+u+2\widetilde{b_s}+1;\\
&y_1=\epsilon+v+2\widetilde{a_1}+r-1,\quad \dots,\quad   y_r=\epsilon+v+2\widetilde{a_r}.
\end{align*}
for some integers $\widetilde{a_1} \geq \dots \geq \widetilde{a_r} \geq 0$, $\widetilde{b_1} \geq \dots \geq \widetilde{b_s} \geq 0$. For each possibility of $\sigma$ satisfying \eqref{BP-sigma}, the PRV component \cite{PRV} of
$$
E_{(\frac{r-s}{2}, \dots, \frac{r-s}{2})+(2\widetilde{a_1}+\epsilon, \dots, 2\widetilde{a_r}+\epsilon, 0, \dots, 0, -2\widetilde{b_s}-\eta, \dots, -2\widetilde{b_1}-\eta)} \otimes E_{\sigma \rho - \rho_c}
$$
appearing in $X(r,s;\epsilon,\eta)$ (c.f. \eqref{sp2nktypes}) and $\mathrm{Spin}_G^{\pm}$ (c.f. \eqref{spin-Weyl}) respectively contributes a copy of $E_{\tau}$ in $H_D(X(r,s;\epsilon,\eta))$.
In order to determine the sign of $E_{\tau}$ in $\mathrm{DI}(X(r,s;\epsilon,\eta))$,
it suffices to study the parity of $l(\sigma)$ by Lemma \ref{lemma-even-odd-spinG}.

Note that for all $\sigma$ satisfying \eqref{BP-sigma}, $\rho - \sigma \rho = \langle \Phi_{\sigma} \rangle$ is a sum of roots in $\Delta^+(\mathfrak{p},\mathfrak{t})$ by Lemma \ref{lemma-Phiw}. Therefore, by \eqref{zeta-k-p} and \eqref{spin-w}, we have
$$l(\sigma) \equiv |\Phi_{\sigma}| \equiv \langle \rho-\sigma\rho, \zeta \rangle \quad ({\rm mod}\ 2).$$
On the other hand, a direct calculation of $\rho - \sigma \rho$ using \eqref{BP-sigma} gives
$$
\langle \rho-\sigma\rho, \zeta \rangle=j_1^\prime+\cdots+j_v^\prime+y_1+\cdots+y_r\equiv M(\tau) \quad ({\rm mod}\ 2).
$$
Therefore, Lemma \ref{lemma-even-odd-spinG} implies that \emph{each} $E_{\tau}$ in $H_D(X(r,s; \epsilon, \eta))$ has the \emph{same} sign $(-1)^{l(\sigma)}=(-1)^{M(\tau)}$.


Still assume that $\tau$ is special. By Lemma 3.2 of \cite{BP15}, \emph{each} $E_{-w_0^K\tau}$ in $H_D(X(r,s; \epsilon, \eta))$ has the \emph{same} sign, which equals to $(-1)^{\frac{n(n+1)}{2}}$ times the sign of $E_{\tau}$ in $H_D(X(s,r; \eta, \epsilon))$. This sign turns out to be $(-1)^{N(\tau)}$.

\medskip
To sum up, when passing from $H_D(X(r,s; \epsilon, \eta))$ to ${\rm DI}(X(r,s; \epsilon, \eta))$, no cancellation happens. For each special $\widetilde{K}$-type $E_{\tau}$, the sign of $E_{\tau}$ (resp., $E_{-w_0^K\tau}$) in ${\rm DI}(X(r,s; \epsilon, \eta))$ is $(-1)^{M(\tau)}$ (resp., $(-1)^{N(\tau)}$). The multiplicity of $E_{\tau}$, which can be zero, is known from Proposition 3.4 and Theorem 3.7 of \cite{BP15}.

\begin{example} \label{eg-sp12}
Let us consider $Sp(12, \bbR)$ and take $k=2$. Then $\Lambda=(4,3,2,1,0,-1)$, $r+s=4$ and $u+v=2$. We list all the four special $\widetilde{K}$-types as follows:
\begin{itemize}
\item[$\bullet$] $\tau_1=(4, 3, 2, 1, 0, -1)- \rho_c=(\frac{3}{2}, \frac{3}{2}, \frac{3}{2}, \frac{3}{2}, \frac{3}{2}, \frac{3}{2})$, $u=2$, $v=0$.
\item[$\bullet$] $\tau_2=(3, 2, 1, 0, -1, -4)- \rho_c=(\frac{1}{2}, \frac{1}{2}, \frac{1}{2}, \frac{1}{2}, \frac{1}{2}, -\frac{3}{2})$, $u=1$, $v=1$, $j_1'=2$.
\item[$\bullet$] $\tau_3=(4, 2, 1, 0, -1, -3)- \rho_c=(\frac{3}{2}, \frac{1}{2}, \frac{1}{2}, \frac{1}{2}, \frac{1}{2}, -\frac{1}{2})$, $u=1$, $v=1$, $j_1'=1$.
\item[$\bullet$] $\tau_4=(2, 1, 0, -1, -3, -4)- \rho_c=(-\frac{1}{2}, -\frac{1}{2}, -\frac{1}{2}, -\frac{1}{2}, -\frac{3}{2}, -\frac{3}{2})$, $u=0$, $v=2$, $j_1'=2$, $j_2'=1$.
\end{itemize}
Now the sign of each $E_{\tau_i}$ and $E_{-w_0^K\tau_i}$ in $H_D(X(r,s; \epsilon, \eta))$ is given in Table \ref{table-sign-X-Sp12}. For instance, it reads that the sign of $E_{\tau_1}$ in $H_D(X(3, 1; \epsilon, \eta))$ (if it does occur) is $(-1)^{\epsilon+1}$.
\end{example}

\begin{table}
\centering
\caption{Signs of the $\widetilde{K}$-types in $H_D(X(r,s; \epsilon, \eta))$ of $Sp(12, \bbR)$}
\begin{tabular}{l|c|c|c|c||c|c|c|c|}
 &   $\tau_1$ &$\tau_2$  & $\tau_3$ & $\tau_4$ & $-w_0^K\tau_1$ & $-w_0^K\tau_2$ & $-w_0^K\tau_3$ & $-w_0^K\tau_4$\\
\hline
$r=4, s=0$ & $0$ & $0$ & $1$ & $1$ & $1$ & $1$ & $0$ & $0$\\
\hline
$r=3, s=1$ & $\epsilon+1$ & $\epsilon$ & $\epsilon+1$ & $\epsilon$ & $\eta+1$ & $\eta$ & $\eta+1$ & $\eta$ \\
\hline
$r=2, s=2$ & $1$ & $1$ & $0$ & $0$ & $0$ & $0$ & $1$ & $1$ \\
\hline
$r=1, s=3$ & $\epsilon$ & $\epsilon+1$ & $\epsilon$ & $\epsilon+1$ & $\eta$ & $\eta+1$ & $\eta$ & $\eta+1$ \\
\hline
$r=0, s=4$ & $0$ & $0$ & $1$ & $1$ & $1$ & $1$ & $0$ & $0$
\end{tabular}
\label{table-sign-X-Sp12}
\end{table}

\begin{thm}\label{thm-X-DI-sum}
Each of the following virtual $(\frg, K)$ module has zero Dirac index:
\begin{itemize}
\item[a)] $\displaystyle \sum_{\epsilon + \eta \equiv \delta \ (\mathrm{mod}\ 2)} X(2j-1, 2k+1-2j; \epsilon, \eta)$, for any $1\leq j\leq k-1$ and $\delta \in \{0,1\}$;
\item[b)] $\displaystyle \sum_{\epsilon} X(2k, 0; \epsilon, 0)+ \sum_{\eta} X(0, 2k; 0, \eta)+\sum_{j=1}^{k-1}\sum_{\epsilon + \eta \equiv \delta \ (\mathrm{mod}\ 2)} X(2j, 2k-2j; \epsilon, \eta)$ for $\delta \in \{0,1\}$.
\end{itemize}
In each case above, both $\epsilon$ and $\eta$ run over $\{0, 1\}$.
\end{thm}
\begin{proof}
Fix any special $\widetilde{K}$-type $E_{\tau}:=w\Lambda -\rho_c$ with $w\Lambda$ given by \eqref{special}. We will show that $E_{\tau}$ has coefficient zero in the Dirac index of each of the above virtual $(\frg, K)$ module. One can draw the same conclusion for each non-special $E_{\tau}$. Thus the desired conclusion follows.

(a) By Theorem 3.7(2II) of \cite{BP15}, when $\delta\equiv n+1 \ ({\rm mod}\ 2)$, each summand has zero Dirac cohomology. Thus the conclusion is trivial. Now assume that $\delta\equiv n\ ({\rm mod}\ 2)$. By Theorem 3.7(2I) of \cite{BP15}, the coefficient of $E_{\tau}$ in ${\rm DI}(X(2j-1, 2k+1-2j; \epsilon, \eta))$ is
$$
(-1)^{M(\tau)}{k-1 \choose j-1}=(-1)^{j_1^\prime + \cdots + j_v^\prime  + v  +j-1+ \epsilon}{k-1 \choose j-1}.
$$
The total sum is zero when $\epsilon$ and $\eta$ run over $\{0, 1\}$ such that $\epsilon + \eta \equiv \delta \ (\mathrm{mod}\ 2)$.

(b) By Proposition 3.4 of \cite{BP15}, the coefficient of $E_{\tau}$ in ${\rm DI}(X(2k, 0; 0, 0) + X(2k, 0; 1, 0))$ is
$$
(-1)^{j_1^\prime + \cdots + j_v^\prime  + k(2k-1)}=(-1)^{j_1^\prime + \cdots + j_v^\prime  + k}.
$$
Similarly,  the coefficient of $E_{\tau}$ in ${\rm DI}(X(0, 2k; 0, 0)+X(0, 2k; 0, 1))$ is
$$
(-1)^{j_1^\prime + \cdots + j_v^\prime}.
$$
Fix $1\leq j\leq k-1$, by Theorem 3.7(1) of \cite{BP15}, the coefficient of $E_{\tau}$ in
$$
{\rm DI}(\sum_{\epsilon + \eta \equiv \delta \ (\mathrm{mod}\ 2)} X(2j, 2k-2j; \epsilon, \eta))
$$
is equal to
$$
(-1)^{j_1^\prime + \cdots + j_v^\prime +j(2j-1)}
{k\choose j}
=(-1)^{j_1^\prime + \cdots + j_v^\prime +j}{k\choose j}.
$$
Note that ${k-1\choose j}+{k-1\choose j-1}={k\choose j}$.
Therefore, the total sum is
$$
(-1)^{j_1^\prime + \cdots + j_v^\prime}
\left[1 + \sum_{j=1}^{k-1}(-1)^{j}{k \choose j} +  (-1)^{k} \right]
=(-1)^{j_1^\prime + \cdots + j_v^\prime}(1-1)^k=0.
$$
\end{proof}

\begin{remark}
In \cite{BT}, Barbasch and Trapa studied the number of \emph{stable combinations} of special unipotent representations
whose annihilator is equal to the closure of a complex special nilpotent orbit $\mathcal{O}$.

For $\pi_u = X(r,s;\epsilon,\eta)$, $\mathcal{O}$ corresponds to the partition $[2^{2k}1^{2n-2k}]$,
and its \emph{Lusztig Spaltenstein dual} $\mathcal{O}^{\vee}$ corresponds to the partition $[2n-2k+1, 2k-1, 1]$
in $SO(2n+1,\mathbb{C})$.
Then the main theorem of \cite{BT} implies that there are
$3$ (resp., $4$) such stable combinations attached to $\mathcal{O}^{\vee}$ if $k = 1$ (resp., if $k > 1$). Note that this does not necessarily
account for all stable combinations: In general,
one needs to take \emph{all} nilpotent orbits in the same \emph{special piece} as $\mathcal{O}^{\vee}$ into account.

Since a necessary condition for a linear combination for some (non-trivial) special unipotent representations
to be stable is that its Dirac index is zero, we conjecture that the sum of the linear combinations in
Theorem \ref{thm-X-DI-sum}(a), and the linear combinations in Theorem \ref{thm-X-DI-sum}(b) are the
$3$ (if $k=1$) or $4$ (if $k > 1$) stable combinations specified in the above paragraph. In particular,
for $\mathcal{O} = [2^2]$ in $Sp(4,\mathbb{R})$, we obtain the $3$ stable combinations
attached to $\mathcal{O}^{\vee} = [3 \ 1^2]$ given in \cite[Example 2.2]{BT}.
\end{remark}

\subsection{Dirac index of $\pi_u = X^\prime(r,s;\epsilon,\eta)$}
Now consider $Sp(2n, \bbR)$ with $n$ odd. Let $r+s=2k=n+1$, where $r, s$ are non-negative integers. Then there is another family of unipotent representations $X^\prime(r,s; \epsilon, \eta)$ studied in \cite{BP15}. All of them have infinitesimal character
$$
\rho_c=(k-1, k-2, \dots, -k+1).
$$
Moreover, one has that
$$
X^\prime(2k, 0; 0, 0)\cong X^\prime (2k-1, 1; 1, 0), \quad
X^\prime(0, 2k; 0, 0)\cong X^\prime (1, 2k-1; 0, 1).
$$
Now we assume that both $p$ and $q$ are positive. Then $(\epsilon, \eta)$ can be $(0, 0)$, $(0, 1)$ or $(1, 0)$.
Adopting the setting of \cite[Theorem 3.8]{BP15}, let us determine the sign of the trivial $\widetilde{K}$-type $E_0$ in $H_D(X^\prime(r,s; \epsilon, \eta))$:
\begin{itemize}
\item[$\bullet$] $(\epsilon, \eta)=(1, 0)$. Then
$$
\sigma\rho=(2b_1+s-1, \dots, 2b_{s-1}+1, -2a_r-1, \dots, -2a_1-r)
.$$
Thus
    $$
    \langle \rho-\sigma\rho, \zeta\rangle=(2a_r+1)+ \cdots+ (2a_1+r).
    $$
    Therefore, the sign is $(-1)^{\frac{r(r+1)}{2}}$.

\item[$\bullet$] $(\epsilon, \eta)=(0, 0)$. Then
$$
\sigma\rho=(2b_1+s-1, \dots, 2b_{s-1}+1, -2c, -2a_{r-1}-1, \dots, -2a_1-r+1),
$$
where $c$ is either $a_r$ or $-b_s$. Thus
$$
\langle \rho-\sigma\rho, \zeta\rangle
\equiv (2a_{r-1}+1)+ \cdots+ (2a_1+r-1) \quad ({\rm mod}\ 2).
$$
Therefore, the sign is $(-1)^{\frac{r(r-1)}{2}}$.

\item[$\bullet$] $(\epsilon, \eta)=(0, 1)$. Then
$$
\sigma\rho=(2b_1+s, \dots, 2b_s+1, -2a_{r-1}-1, \dots, -2a_1-r+1).
$$
Thus
$$
\langle \rho-\sigma\rho, \zeta\rangle=(2a_{r-1}+1)+ \cdots+ (2a_1+r-1).
$$
Therefore, the sign is $(-1)^{\frac{r(r-1)}{2}}$.
\end{itemize}

As in Theorem \ref{thm-X-DI-sum}, we give a linear combination of $(\frg, K)$-modules that is conjectually a stable combination:
\begin{thm}\label{thm-Xprime-DI-sum}
Let $G=Sp(2n, \bbR)$ with $n$ odd. The following virtual $(\frg, K)$ module has zero Dirac index:
$$
X^\prime(0, 2k; 0, 0) + \sum_{j=1}^{k-1}X'(2j, 2k-2j; 0, 0) + X'(2k, 0; 0, 0).
$$
\end{thm}
\begin{proof}
Let us figure out the coefficient of $E_0$ in the Dirac index of each $X'(2j, 2k-2j; 0, 0)$ using Theorem 3.8 of \cite{BP15}. Indeed,

\begin{itemize}

\item[$\bullet$] For $X^\prime(0, 2k; 0, 0)\cong X^\prime (1, 2k-1; 0, 1)$, the coefficient is  $(-1)^0{k-1 \choose 0}=(-1)^0{k \choose 0}$.

\item[$\bullet$] For $X'(2j, 2k-2j; 0, 0)$, $1\leq j \leq k-1$, the coefficient is $(-1)^{\frac{2j(2j-1)}{2}}{k \choose j}=(-1)^j{k \choose j}$.

\item[$\bullet$] For $X^\prime(2k, 0; 0, 0)\cong X^\prime (2k-1, 1; 1, 0)$, the coefficient is $(-1)^{\frac{2k(2k-1)}{2}}{k-1 \choose k-1}=(-1)^k{k \choose k}$.
\end{itemize}
Therefore, the total sum is
$$
(-1)^0{k \choose 0}+\sum_{j=1}^{k-1}(-1)^j{k \choose j}+(-1)^k{k \choose k}
=(1-1)^k=0.
$$
\end{proof}

\section{Dirac index of unipotent representations of $SO^*(2n)$} \label{sec-sostar2n}
Let $G$ be $SO^*(2n)$.
We fix a Vogan diagram for $\frs\fro^*(2n)$ as in Fig.~\ref{Fig-SOstar2n-Vogan}, where the simple roots are $e_1-e_2$, $e_2-e_3$, $\dots$, $e_{n-3}-e_{n-2}$, $e_{n-2}-e_{n-1}$, $e_{n-1}-e_n$ and $e_{n-1} + e_n$ when enumerated  from left to right, with $e_{n-1}+e_n$ being dotted. Therefore, $\zeta=(\frac{1}{2}, \dots, \frac{1}{2})$.

\begin{figure}[H]
\centering
\scalebox{0.80}{\includegraphics{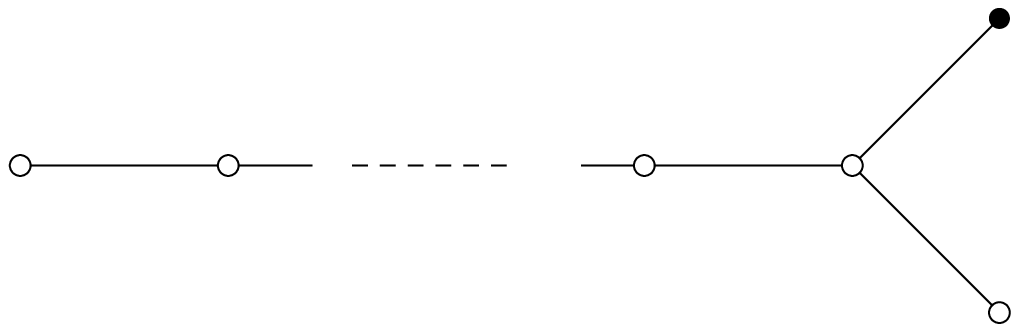}}
\caption{The Vogan diagram for $\frs\fro^*(2n, \bbR)$}
\label{Fig-SOstar2n-Vogan}
\end{figure}

By fixing this Vogan diagram, we have actually fixed
$$
\Delta^+(\frk, \frt)=\{e_i-e_j \mid 1\leq i<j \leq n\},
$$
and chosen
$$
\Delta^+(\frp, \frt)=\{e_i+e_j \mid 1\leq i<j \leq n\}.
$$

\subsection{Unipotent representations of $SO^*(2n)$}

This subsection aims to classify
all the unipotent representations of $SO^*(2n)$ whose infinitesimal characters satisfy Theorem \ref{thm-HP}. Similar to \cite{BP15}, our basic tool is theta correspondence \cite{Ho1,Ho2,Ho3}.

\begin{thm}
Let $G = SO^*(2n)$. Then the unipotent representations of $G$ whose infinitesimal characters satisfy Theorem \ref{thm-HP}
are obtained from theta lifts of the trivial representation in $Sp(r,s)$ with $2(r + s) \leq n-1$.
\end{thm}
\begin{proof}
By the condition of the infinitesimal character $\Lambda$ of $\pi_u$ given by Theorem \ref{thm-HP},
the coordinates of $\Lambda$ must be integers. Therefore, the associated variety of $Ann_{U(\mathfrak{g})}(\pi_u)$ must be
a \emph{special} nilpotent orbit, and $\pi_u$ must be a \emph{special unipotent representation}.
Therefore, $\Lambda = h^{\vee}/2$, where $h^{\vee}$ is the semisimple element of a Jacobson-Morozov triple of a special
(more precisely, even) nilpotent orbit.

Putting in the extra condition that $w\Lambda$ is $\Delta^+(\mathfrak{k},\mathfrak{t})$-dominant, $w\Lambda$
can only be of the form
$$\Lambda_k' = (n-k-1, \dots, k, \dots, 1, 0, -1, \dots, -k)$$
for $0 \leq 2k < n$.

Suppose now the infinitesimal character of $\pi_u$ is equal to $\Lambda = \Lambda_k'$. Then $$
{\rm AV}(Ann_{U(\mathfrak{g})}(\pi_u)) = \overline{\mathcal{O}_k},
$$
where $\mathcal{O}_k$ is the nilpotent orbit in $\mathfrak{so}(2n,\mathbb{C})$ corresponding to the partition
$[2^{2k}1^{n-2k}]$. By \cite{McG}, the number of such special unipotent representations is equal to $k+1$.

For each $r+s = k$,  the theta lift of the trivial representation of $Sp(r,s)$
to $SO^*(2n)$ has infinitesimal character $\Lambda_k$ \cite{Prz}. Moreover, its associated variety is equal
to the closure of the $K$-nilpotent orbit $\mathcal{O}_{r,s}$ with partition $[(+-)^{2r} (-+)^{2s} +^{n-k} -^{n-k}]$.
See \cite{LM}. Therefore, the $k+1$ representations constructed in this way are distinct, and
satisfy ${\rm AV}(Ann_{U(\mathfrak{g})}(\pi_u)) = \overline{\mathcal{O}_k}$. This exhausts all
possibilities of unipotent representations.
\end{proof}
We denote the representation obtained by theta lift of the trivial representation
of $Sp(r,s)$ by $X(r,s)$. In the next subsection, we will study the Dirac cohomology and Dirac index
of $X(r,s)$. It is worth noting that for $n \leq 6$, all $X(r,s)$ can be realized as $\aq(\lambda)$-modules.
The first example of $X(r,s)$ not being an $\aq(\lambda)$-module is $X(1,1)$ in $SO^*(14)$.

\subsection{Dirac index of $X(r,s)$}
We will proceed as in \cite{BP15} and the previous section to find the Dirac index of $X(r,s)$
for $2k = 2(r+s) \leq n -1$. Firstly, recall that $X(r,s)$ has infinitesimal character
$\Lambda_k = (n-k-1, \dots, 1, 0, -1, \dots, -k)$ and the highest weights of its $K$-types are
\begin{equation}\label{KTypes-Xpq}
(r-s, \dots, r-s) + (a_1, a_1, \dots, a_r, a_r, \overbrace{0, \dots, 0}^{n-2k}, -b_s, -b_s, \dots, -b_1, -b_1)
\end{equation}
with $a_1 \geq \dots \geq a_r$ and $b_1 \geq \dots \geq b_s$ all being non-negative integers.
Note that these $K$-types appear in $X(r,s)$ with multiplicity one.

Now let us compute $H_D(X(r,s))$. Note that $\Lambda_k'$ is conjugate to
$$
\Lambda=(n-k-1, \dots, k+1, k, \dots, 1, 0, -1, \dots,  -k).
$$
By Theorem \ref{thm-HP}, any $\widetilde{K}$-type $E_{\tau}$ occurring in $H_D(X(r,s))$ must bear the form
$$
\tau=x\Lambda-\rho_c, \quad \mbox{ for some } x \in W(\frg, \frt)^1.
$$
Let $u, v$ be two non-negative integers such that $u+v=n-2k-1$.  Then we must have
\begin{equation}\label{xLambda-Xpq}
x\Lambda=(i_1, \dots, i_u, k, \dots, 1, 0, -1, \dots, -k, -j_v, \dots, -j_1),
\end{equation}
where $i_1> \cdots > i_u$, $j_1> \dots> j_v$ and
$$
\{i_1, \dots, i_u, j_1, \dots, j_v \}=\{k+1, \dots, n-k-1\}.
$$
We put $i_t'=i_t-k$ and $j_t'=j_t-k$. Then
$$
\{i'_1, \dots, i'_u, j'_1, \dots, j'_v \}=\{1, \dots, n-2k-1\}.
$$

Now it boils down to solve the equation
\begin{equation}\label{Dirac-SOstar-unipotent}
\sigma\rho -\rho_c + w_0^K \mu =w\tau, \quad w\in W(\frk, \frt), \ \sigma\in W(\frg, \frt)^1,
\end{equation}
where $\mu$ is a $K$-type of $X(r,s)$ as described in \eqref{KTypes-Xpq}. Similar to Section 3 of \cite{BP15}, we have
\begin{equation}\label{Xpq-sigmarho}
\begin{aligned}
\sigma\rho =&(b_1+u+2s, b_1+u+2s-1, \dots, b_s+u+2, b_s+u+1, i'_1, \dots, i'_u, 0, -j'_v, \\
&\dots, -j_1', -a_r-v-1, -a_r-v-2, \dots, -a_1-v-2r+1, -a_1-v-2r).
\end{aligned}
\end{equation}
Let us arrange  $n-2k, \dots, n-1$ into the following $k$ pairs of consecutive integers:
\begin{equation}\label{k-pair}
n-2k, n-2k+1 \mid \mid  n-2k+2, n-2k+3 \mid \mid \quad  \cdots  \quad \mid \mid  n-2, n-1.
\end{equation}

Fix $\tau=x\Lambda -\rho_c$ with $x\Lambda$ being given by \eqref{xLambda-Xpq}. Note that
$$
v+2r\leq  (n-2k-1)+2k  \leq n-1, \quad u+2s\leq (n-2k-1)+2k  \leq n-1.
$$
Thus for any choice of $r$ pairs of consecutive integers from \eqref{k-pair}, there is a unique solution to \eqref{Dirac-SOstar-unipotent} in terms of $a_1, \dots, a_r$. Therefore, the multiplicity of $E_{\tau}$ in $H_D(X(r,s))$ is ${k\choose r}$. Moreover, it follows from \eqref{Xpq-sigmarho} that
$$
\langle \rho-\sigma\rho, \zeta \rangle=j'_1+\cdots+j'_v +(a_r+v+1)+(a_r+v+2)+\cdots + (a_1+v+2r-1) + (a_1+v+2r).
$$
Thus the sign of $E_{\tau}$ in $H_D(X(r,s))$ is
\begin{equation}\label{sign-tau-Xpq}
(-1)^{j'_1+\cdots+j'_v+r}.
\end{equation}

The above discussion leads to the following.
\begin{thm}\label{thm-Dirac-Xpq}
Let $G$ be $SO^*(2n)$. Let $k$ be a positive integer such that $2k+1\leq n$. The Dirac cohomology of $X(r, k-r)$ is
$$
{k\choose r}\sum_{v=0}^{n-2k-1}\sum_{j_1'> \cdots> j'_v}E_{(i'_1, \dots, i'_u, k, \dots, 1, 0, -1, \dots, -k, -j'_v, \dots, -j'_1)-\rho_c},
$$
where $\{j_1', \dots,  j'_v\}$ run over the subsets of $\{1, 2, \dots, n-2k-1\}$ with cardinality $v$, and $\{i'_1, \dots, i'_u\}$ is the complementary set.
\end{thm}

Using \eqref{sign-tau-Xpq}, one easily deduces the following.

\begin{cor}\label{cor-thm-Dirac-Xpq}
Let $k$ be a positive integer such that $2k+1\leq n$.
The virtual $(\frg, K)$ module $\sum_{r=0}^{k} X(r, k-r)$ of $SO^*(2n)$ has zero Dirac index.
\end{cor}

\section{Dirac index for general modules}  \label{sec-general}
We now study the Dirac index for all unitary representation given in Theorem \ref{thm-Vog84} for $G = Sp(2n,\mathbb{R})$ and $SO^*(2n)$.
Firstly, we describe the $\theta$-stable parabolic subalgebras $\mathfrak{q} = \mathfrak{l} + \mathfrak{u}$
of $\mathfrak{g}$ using $H \in \mathfrak{t}^*$ of the form (c.f. Section 3 of \cite{Pr})
\begin{equation} \label{eq-hparabolic}
\begin{aligned}
H = (&\overbrace{l, \dots, l}^{p_l}, \  \overbrace{l-1, \dots, l-1}^{p_{l-1}}, \  \cdots, \  \overbrace{1, \dots, 1}^{p_1},\  \overbrace{ 0, \dots, 0}^{m}, \\
&\underbrace{-1, \dots, -1}_{q_1}, \  \cdots,\  \underbrace{-l+1,\  \dots,\  -l+1}_{q_{l-1}},\  \underbrace{-l, \dots, -l}_{q_l})
\end{aligned}
\end{equation}
such that
$$\mathfrak{l}_0 = \mathfrak{u}(p_l, q_l) \oplus \mathfrak{u}(p_{l-1}, q_{l-1}) \oplus \dots \oplus \mathfrak{u}(p_1, q_1) \oplus \mathfrak{g}_m',$$
where $\mathfrak{g}_m'$ is of the same type as $\mathfrak{g}$ with rank $m$.
And we choose $(\Delta^+)^\prime(\mathfrak{g},\mathfrak{t})$ such that
\begin{align*}
\rho_{sp}^\prime := (&\overbrace{n, n - 1, \dots , n - p_l + 1}^{p_l\ terms}, \, \overbrace{n - (p_l + q_l), \dots , n - (p_l + q_l) - p_{l-1} + 1}^{p_{l-1}\ terms}, \, \cdots , \\
&\overbrace{m + p_1 + q_1, \dots , m + q_1 + 1}^{p_1\ terms}, m,m - 1, \dots , 1
,\underbrace{-(m + 1), -(m + 2), \dots ,-(m + q_1)}_{q_1\ terms},\\
&\cdots, \underbrace{-(n - (p_l + q_l) - (p_{l-1} + q_{l-1}) + 1), \dots ,-(n - (p_l + q_l) - p_{l-1})}_{q_{l-1}\ terms}, \\
&\underbrace{-(n - (p_l + q_l) + 1), \dots, -(n - p_l)}_{q_l\ terms})
\end{align*}
for $G = Sp(2n,\mathbb{R})$, and
$$\rho_{so}^\prime := \rho_{sp}^\prime - (\overbrace{1, \dots, 1}^{p_l + \dots + p_1},\  \overbrace{ 1, \dots, 1}^{m},
\overbrace{-1, \dots, -1}^{q_1 + \dots + q_l})$$
for $G = SO^*(2n)$.

We now describe the $(\mathfrak{l}, L \cap K)$-modules $Z$ appearing in Theorem \ref{thm-Vog84}. They are all of the form
\begin{equation} \label{eq-Z}
Z(\lambda_l,\dots,\lambda_1,\pi_u) := \mathbb{C}_{\lambda_l} \boxtimes \dots \boxtimes \mathbb{C}_{\lambda_1} \boxtimes \pi_u,
\end{equation}
where each $\mathbb{C}_{\lambda_i}$ is a unitary character of $\mathfrak{u}(p_i,q_i)$, and $\pi_u$ is a
weakly unipotent representation. It suffices to
focus on those $\pi_u$ covered in Sections \ref{sec-sp2nr} and \ref{sec-sostar2n} for the study of $\mathrm{DI}(\mathcal{L}_S(Z))$.

\begin{definition} \label{def-chains}
Let $G = Sp(2n,\mathbb{R})$ or $SO^*(2n)$, and $\mathfrak{q} = \mathfrak{l} + \mathfrak{u}$ be the $\theta$-stable
parabolic subalgebra of $\frg$ defined by the element $H$ in \eqref{eq-hparabolic}. For each $(\mathfrak{l}, L \cap K)$ module
$Z = Z(\lambda_l,\dots,\lambda_1,\pi_u)$ given by \eqref{eq-Z}, the \emph{chains attached to $Z$} are defined by
$$Z(\lambda_l,\dots,\lambda_1,\pi_u)  \leftrightsquigarrow (\mathcal{C}_l^{\lambda_l}, \dots, \mathcal{C}_1^{\lambda_1}, \mathcal{C}_0)$$
where
\begin{align*}
\mathcal{C}_i^{\lambda} := (\mu_i + \lambda, \mu_i - 1 + \lambda, \dots, \mu_i - (p_{i} + q_{i}) + 1 + \lambda)^{p_i, q_i}, \quad \mu_i = \rho_1 - \sum_{t = i+1}^l (p_t + q_t)
\end{align*}
for $1 \leq i \leq l$, and $\mathcal{C}_0 := \Lambda_u$ is the $\mathfrak{g}_m'$-dominant infinitesimal character of the unipotent representation $\pi_u$.
\end{definition}
One reason for introducing the above notation is that the chains of $Z$ gives the infinitesimal character of $\mathcal{L}_S(Z)$.
Moreover, it gives us an easy way to determine the (weakly) goodness or (weakly) fairness of $\mathcal{L}_S(Z)$.

\begin{lemma} \label{lem-rephrase}
Let $Z(\lambda_l,\dots,\lambda_1,\pi_u)$ be an $(\mathfrak{l}, L \cap K)$-module corresponding to the chains
$(\mathcal{C}_l^{\lambda_l}, \dots, \mathcal{C}_1^{\lambda_1}, \mathcal{C}_0)$ given in Definition \ref{def-chains}.
Then $\mathcal{L}_S(Z(\lambda_l,\dots,\lambda_1,\pi_u))$ is in the good range if and only if
$$\lambda_i - \lambda_{i-1} > -1 \ \text{for}\ \ i = l, l-1, \dots, 2; \quad \text{and}\quad \lambda_1 + (m+1) > u_m $$
In other words, $\mathcal{L}_S(Z(\lambda_l,\dots,\lambda_1,\pi_u))$ is good if and only if
\begin{center}
(smallest entry of the $i^{th}$-chain) $>$ (largest entry of the $(i-1)^{th}$-chain), \quad $i = l, \dots, 1$.
\end{center}
And $\mathcal{L}_S(Z(\lambda_l,\dots,\lambda_1,\pi_u))$ is  weakly good if we replace the above strict inequalities with $\geq$.

Moreover, $\mathcal{L}_S(Z(\lambda_l,\dots,\lambda_1,\pi_u))$ is in the fair range if and only if
$$
\mu_l- \frac{(p_l + q_l)(p_l + q_l -1)}{2}
> \cdots >   \mu_1 - \frac{(p_1 + q_1)(p_1 + q_1 -1)}{2} > 0.
$$
In other words, $\mathcal{L}_S(Z(\lambda_l,\dots,\lambda_1,\pi_u))$ is fair if and only if
\begin{center}
(average value of the entries of $\mathcal{C}_l^{\lambda_l}$) $> \dots >$ (average value of the entries of $\mathcal{C}_1^{\lambda_1}$) $> 0$
\end{center}
And $\mathcal{L}_S(Z(\lambda_l,\dots,\lambda_1,\pi_u))$ is weakly fair if we replace the above strict inequalities with $\geq$.
\end{lemma}

\begin{example} \label{eg-so14-1}
Let $G = SO^*(14)$ and $\mathfrak{q} = \mathfrak{l} + \mathfrak{u}$ is determined by $H = (2,1,0,0,0,-1,-2)$,
so that $\mathfrak{l}_0 = \mathfrak{u}(1,1) + \mathfrak{u}(1,1)+\mathfrak{so}^*(6)$.
The chains corresponding to $Z(0, 0, \pi_u)$ with $\pi_u = X(1,0)$ are
$$\mathcal{C}_2^0 = (6,5)^{1,1}, \qquad \mathcal{C}_1^0 = (4,3)^{1,1}, \qquad \mathcal{C}_0 = (1,1,0).$$
So $\mathcal{L}_S(Z(0,0,\pi_u))$ is in the good range by Lemma \ref{lem-rephrase}.

The chains corresponding to $Z(-2, -1, \pi_u)$ are
$$\mathcal{C}_2^{-2} = (4,3)^{1,1}, \qquad \mathcal{C}_1^{-1} = (3,2)^{1,1}, \qquad \mathcal{C}_0 = (1,1,0).$$
Thus $\mathcal{L}_S(Z(-2,-1,\pi_u))$ is in the weakly good range.\hfill\qed
\end{example}

\begin{example}\label{eg-sp22-1}
Let $G = Sp(22,\mathbb{R})$, and let $\mathfrak{q}$ be the $\theta$-stable parabolic subalgebra of $\frg$ defined by $H = (1,0,0,0,0,0,0,0,0,-1,-1)$. The chains corresponding to $\mathcal{L}_S(Z(-6,X(4, 2; 0,0)))$ are
$$
\mathcal{C}_1^{-6} = (5,4,3)^{1,2}, \quad \mathcal{C}_0=(5,4,3,2,1,0,-1,-2).
$$
The module $\mathcal{L}_S(Z(-6,X(4, 2; 0,0)))$ is not weakly good, but it is fair.  \hfill \qed
\end{example}

We begin our study of the Dirac index of $\mathcal{L}_S(Z(\lambda_l,\dots,\lambda_1,\pi_u))$ by looking at the special
case when $\lambda_l = \dots = \lambda_1 = 0$. Note that for all $\pi_u$ in Section \ref{sec-sp2nr} and Section \ref{sec-sostar2n}, $\mathcal{L}_S(Z(0,\dots,0,\pi_u))$ is always in good range by Lemma \ref{lem-rephrase}. Using Proposition 4.1 of \cite{DW4}, one easily deduce the following.

\begin{thm} \label{thm-DH}
Let $G$ be $Sp(2n,\mathbb{R})$ or $SO^*(2n)$, with $\theta$-stable parabolic $\mathfrak{q} = \mathfrak{l} + \mathfrak{u}$
and $Z(0,\dots,0,\pi_u)$ be  defined as above.
Then $\mathrm{DI}(\mathcal{L}_S(Z(0,\dots,0,\pi_u)))$ is nonzero if and only if $\pi_u$ has nonzero Dirac index.

More precisely, let $\mathfrak{g}_A := \mathfrak{u}(p,q)$ where $p := \sum_{t=1}^l p_t$,  $q := \sum_{t=1}^l q_t$,
and $\mathfrak{q}_A = \mathfrak{l}_A + \mathfrak{u}_A$ be a $\theta$-stable parabolic subalgbera of $\mathfrak{g}_A$
with $(\mathfrak{l}_A)_0 := \mathfrak{u}(p_l,q_l) \oplus \dots \oplus \mathfrak{u}(p_1,q_1)$. Suppose
$$\mathrm{DI}(A_{\mathfrak{q}_A}(\mathbb{C}_{n-\frac{p+q+1}{2}})) = \sum_{u,v} \epsilon_{u,v} \widetilde{E}_{(\kappa_u| \kappa_v)},\quad \mathrm{DI}(\pi_u) = \sum_w \delta_w \widetilde{E}_{\kappa_w}.$$
Then
$$\mathrm{DI}(\mathcal{L}_S(Z(0,\dots,0,\pi_u))) = \sum_{v,w} \epsilon_{u,v} \delta_w \widetilde{E}_{(\kappa_u;\kappa_w;\underline{\kappa_v})},$$
where $(\underline{v_1, \dots, v_{\ell}})$ is defined to be $(\underline{v_1, \dots, v_{\ell}}) := (-v_{\ell}, \dots, -v_1)$.
\end{thm}


\begin{example} \label{eg-so14-2}
We continue with Example \ref{eg-so14-1}, where $G = SO^*(14)$, and $Z = Z(0, 0, X(1,0))$.
Then $n - \frac{p+q+1}{2} = 7 - \frac{5}{2} = \frac{9}{2}$, and by Theorem \ref{thm-DI-Aqlambda},
\begin{align*}
\mathrm{DI}(A_{\mathfrak{q}_A}(\mathbb{C}_{\frac{9}{2}}))
= &+ \widetilde{E}_{(\frac{3}{2},-\frac{1}{2}|\frac{1}{2},-\frac{3}{2}) + (\frac{9}{2},\frac{9}{2}|\frac{9}{2},\frac{9}{2})} - \widetilde{E}_{(\frac{1}{2},-\frac{1}{2}|\frac{3}{2},-\frac{3}{2})+ (\frac{9}{2},\frac{9}{2}|\frac{9}{2},\frac{9}{2})}\\
&- \widetilde{E}_{(\frac{3}{2},-\frac{3}{2}|\frac{1}{2},-\frac{1}{2})
+ (\frac{9}{2},\frac{9}{2}|\frac{9}{2},\frac{9}{2})}
 +\widetilde{E}_{(\frac{1}{2},-\frac{3}{2}|\frac{3}{2},-\frac{1}{2})+ (\frac{9}{2},\frac{9}{2}|\frac{9}{2},\frac{9}{2})}\\
= &+ \widetilde{E}_{(6,4|5,3)} - \widetilde{E}_{(5,4|6,3)} - \widetilde{E}_{(6,3|5,4)} +
\widetilde{E}_{(5,3|6,4)}.
\end{align*}
Applying Theorem \ref{thm-Dirac-Xpq} to the unipotent representation $\pi_u = X(1,0)$ of $SO^*(6)$ gives
$$\mathrm{DI}(\pi_u) = - \widetilde{E}_{(1,0,-1)}.$$
So Theorem \ref{thm-DH} implies that
\begin{align*} \mathrm{DI}(\mathcal{L}_S(Z(0,0,\pi_u))) = &- \widetilde{E}_{(6,4|1,0,-1|-3,-5)} + \widetilde{E}_{(5,4|1,0,-1|-3,-6)} \\ &+ \widetilde{E}_{(6,3|1,0,-1|-4,-5)} -\widetilde{E}_{(5,3|1,0,-1|-4,-6)}.
\end{align*}
\hfill\qed
\end{example}

As for the general case when $\mathcal{L}_S(Z(\lambda_l,\dots,\lambda_1,\pi_u))$ is in weakly fair range,
we begin by studying all possible $\lambda_l$, $\dots$, $\lambda_1$ (or equivalently, $\mathcal{C}_l^{\lambda_l}$, $\dots$, $\mathcal{C}_1^{\lambda_1}$, $\mathcal{C}_0$) such that the infinitesimal character satisfies Theorem \ref{thm-HP}.

Indeed, one only needs to focus on the chains $\mathcal{C}_i^{\lambda_i}$ that are \emph{interlaced} with the `unipotent' chain $\mathcal{C}_0$, since the case when $\mathcal{C}_i^{\lambda_i}$ and $\mathcal{C}_j^{\lambda_j}$ are interlaced is identical to the situation of $U(p,q)$ given in \cite{DW4}.
To start with, recall the chain (i.e., the infinitesimal character) corresponding to the unipotent representation
$X(r,s;\epsilon,\eta)$ in $Sp(2n,\mathbb{R})$ is equal to
$$
\mathcal{C}_0 = (n-k, \dots, k+1, k; k-1, k-1, \dots, 1,1,0), \quad \text{where}\ 2k = r+s \leq n;
$$
and the chain for $X'(r,s;\epsilon,\eta)$ in $Sp(2n,\mathbb{R})$ for odd $n$ is of the form
$$
\mathcal{C}_0 = (\frac{n-1}{2}, \frac{n-1}{2}, \dots, 1, 1,0);
$$
and the chain for $X(r,s)$ in $SO^*(2n)$ is of the form
$$
\mathcal{C}_0 = (n-k-1, \dots, k+1; k, k, \dots, 1, 1, 0), \quad \text{where}\ 2k = 2(r+s) < n.
$$
It is obvious that for any weakly fair $\mathcal{L}_S(Z(\lambda_l, \dots, \lambda_1,X'(r,s;\epsilon,\eta)))$ such that
$\mathcal{C}_t^{\lambda_t}$ and $\mathcal{C}_0$ are interlaced, the infinitesimal character
of $\mathcal{L}_S(Z(\lambda_l, \dots, \lambda_1,X'(r,s;\epsilon,\eta)))$ violates Theorem \ref{thm-HP}, and hence it must have
zero Dirac cohomology and Dirac index.

As for the other two types of unipotent representations, in order for the infinitesimal
character of $\mathcal{L}_S(Z(\lambda_l, \dots, \lambda_1,\pi_u))$ to satisfy Theorem \ref{thm-HP},
its chains $\mathcal{C}_l^{\lambda_l}$, $\dots$, $\mathcal{C}_1^{\lambda_0}$, $\mathcal{C}_0$ must be of the form:
\begin{equation} \label{eq-interlaced}
\begin{tabular}{ccccccccccccc}
$(\mathcal{A}$; & $\mathcal{R}_l)^{p_l,q_l}$ &  & $(\mathcal{R}_{l-1})^{p_{l-1},q_{l-1}}$ & $\dots$  & $\dots$ & $(\mathcal{R}_{1})^{p_{1},q_{1}}$ &  &  \tabularnewline
 & $(\mathcal{R}_l;$ \quad \quad \quad & $\mathcal{B}_l$; & $\mathcal{R}_{l-1}$;\quad \quad \quad \quad  & $\mathcal{B}_{l-1}$; & $\dots$; & $\mathcal{R}_{1}$; \quad \quad & $\mathcal{B}_1$; & $\mathcal{K}) = \mathcal{C}_0$,
\end{tabular}
\end{equation}
where
$$\mathcal{K} := \begin{cases} (k-1,k-1, \dots, 1,1,0) &\text{if}\ G = Sp(2n,\mathbb{R})\\ (k, k, \dots, 1,1,0) &\text{if}\ G = SO^*(2n) \end{cases}$$
(we also allow $p_l = q_l = 0$, so that $\mathcal{C}_l$ does not exist in \eqref{eq-interlaced}). For the rest of this section, we study $\mathrm{DI}(\mathcal{L}_S(Z(\lambda_l, \dots, \lambda_1,\pi_u)))$
for $\pi_u = X(r,s;\epsilon,\eta)$ or $X(r,s)$, whose chains are of the form \eqref{eq-interlaced}.

 The following result strengthens Lemma 7.2.18(b) of \cite{Vog81}.

\begin{lemma}\label{lemma-translation}
Let $\lambda_0\in\frh^*$ be dominant integral for $\Delta^+(\frg, \frh)$ which may be singular. Let $F_{\nu}$ be a finite-dimensional $(\frg, K)$ module with extreme weight $\nu$.  Assume that $\lambda_0+\nu$ is dominant for $\Delta^+(\frg, \frh)$, and that
\begin{equation}\label{finite-extreme}
w(\lambda_0 + \nu)=\lambda_0 + \mu,
\end{equation}
where $\mu$ is a weight of $F_{\nu}$ and $w\in W(\frg, \frh)$. Assume moreover that
\begin{equation}\label{nu-dot-alpha}
\langle \nu, \alpha \rangle =0,
\end{equation}
where $\alpha$ is any root in $\Delta^+(\frg, \frh)$ such that $\langle \lambda_0, \alpha\rangle=0$. Then we must have $\mu=\nu$.
\end{lemma}
\begin{proof}
By \eqref{finite-extreme}, we have $\mu=w(\lambda_0+\nu)-\lambda_0$. Thus
\begin{equation}\label{mu-dot-mu}
\langle \mu, \mu \rangle=\langle \lambda_0+\nu, \lambda_0+\nu \rangle - 2 \langle  w(\lambda_0+\nu), \lambda_0 \rangle +\langle\lambda_0, \lambda_0 \rangle.
\end{equation}
On the other hand, by Lemma 6.3.28 of \cite{Vog81},
\begin{equation}\label{w-lambda-nu}
w(\lambda_0 + \nu)=\lambda_0 + \nu -\sum_{\alpha\in\Delta^+(\frg, \frh)} n_{\alpha} \alpha,
\end{equation}
where $n_{\alpha}$ are non-negative real numbers. Therefore, by dominance of $\lambda_0$,
\begin{equation}\label{w-dot-mu}
\langle w(\lambda_0 + \nu), \lambda_0 \rangle \leq \langle \lambda_0 + \nu, \lambda_0 \rangle.
\end{equation}
Substituting \eqref{w-dot-mu} into \eqref{mu-dot-mu}, we get
$$
\langle \mu, \mu \rangle \geq \langle \lambda_0+\nu, \lambda_0+\nu \rangle - 2 \langle  \lambda_0+\nu, \lambda_0 \rangle +\langle\lambda_0, \lambda_0 \rangle=\langle \nu, \nu\rangle.
$$
We must have $\langle \mu, \mu \rangle=\langle \nu, \nu \rangle$ since $\mu$ is a weight of $F_{\nu}$. Thus \eqref{w-dot-mu} must be an equality, and we conclude that $\lambda_0$ is perpendicular to any $\alpha$ in \eqref{w-lambda-nu} whose coefficient $n_{\alpha}$ is positive. Combining  \eqref{finite-extreme} and \eqref{w-lambda-nu}, we have that
$$
\mu=\nu-\sum_{\alpha\in\Delta^+(\frg, \frh), \, n_{\alpha}>0} n_{\alpha} \alpha.
$$
Thus
\begin{align*}
\langle \nu, \nu \rangle=\langle \mu, \mu \rangle &= \langle \nu, \nu \rangle -2 \sum_{\alpha\in\Delta^+(\frg, \frh), \, n_{\alpha}>0} \langle \nu, n_{\alpha} \alpha \rangle +\|\sum_{\alpha\in\Delta^+(\frg, \frh), \ n_{\alpha}>0} n_{\alpha} \alpha\|  \\
&=\langle \nu, \nu \rangle +\|\sum_{\alpha\in\Delta^+(\frg, \frh), \, n_{\alpha}>0} n_{\alpha} \alpha\|,
\end{align*}
where the last step uses \eqref{nu-dot-alpha}. We conclude that
$$
\sum_{\alpha\in\Delta^+(\frg, \frh), \, n_{\alpha}>0} n_{\alpha} \alpha=0.
$$
Thus $\mu=\nu$ as desired.
\end{proof}

\begin{prop} \label{prop-translation}
Let $\mathcal{L}_S(Z(\lambda_l,\dots,\lambda_1,\pi_u))$ be a $(\mathfrak{g},K)$ module
given by Theorem \ref{thm-Vog84} such that its chains $\mathcal{C}_l^{\lambda_l}$, $\dots$, $\mathcal{C}_1^{\lambda_1}$, $\mathcal{C}_0$ satisfy \eqref{eq-interlaced}. Suppose the Dirac index of
$\mathcal{L}_S(Z(0,\dots,0,\pi_u))$ is given in Theorem \ref{thm-DH} by
$$
\mathrm{DI}(\mathcal{L}_S(Z(0,\dots,0,\pi_u))) = \sum_{z} \epsilon_{z} \widetilde{E}_{\kappa_z}.
$$
Then
$$\mathrm{DI}(\mathcal{L}_S(Z(\lambda_l,\dots,\lambda_1,\pi_u))) = \sum_{z} \epsilon_{z} \widetilde{E}_{\kappa_z - (\lambda)},$$
where
\begin{align*}
(\lambda) := (&\overbrace{\lambda_l, \dots, \lambda_l}^{p_l}, \  \overbrace{\lambda_{l-1}, \dots, \lambda_{l-1}}^{p_{l-1}}, \  \dots, \  \overbrace{1, \dots, 1}^{p_1},\  \overbrace{ 0, \dots, 0}^{m}, \\
&\underbrace{-1, \dots, -1}_{q_1}, \  \cdots,\  \underbrace{-\lambda_{l-1},\  \dots,\  -\lambda_{l-1}}_{q_{l-1}},\  \underbrace{-\lambda_l, \dots, -\lambda_l}_{q_l})
\end{align*}
\end{prop}

\begin{proof}
For a fixed $\mathfrak{q} = \mathfrak{l} + \mathfrak{u}$ and a fixed unipotent representation $\pi_u$,
$Z(\lambda_l, \dots, \lambda_1, \pi_u)$ are in the same \emph{coherent family} \cite[Definition 7.2.5]{Vog81}
of virtual $(\mathfrak{l}, L \cap K)$ modules. Therefore, by \cite[Corollary 7.2.10]{Vog81},
the vitural $(\mathfrak{g},K)$ modules
$$
X_{\mathcal{C}_l^{\lambda_l}; \dots; \mathcal{C}_1^{\lambda_1}; \mathcal{C}_0} := \sum_{i}(-1)^i\mathcal{R}^{S-i}(Z(\lambda_l, \dots, \lambda_1, \pi_u))
$$
are in a coherent family. Recall that $\mathcal{C}_t^{\lambda}$ is obtained by adding $\lambda$ to each coordinate of $\mathcal{C}_t$. If $Z(\lambda_l, \dots, \lambda_1, \pi_u)$ corresponds to chains of the form \eqref{eq-interlaced}, it is in the weakly fair range. Thus by Theorem \ref{thm-Vog84},
$$
X_{\mathcal{C}_l^{\lambda_l}; \dots; \mathcal{C}_1^{\lambda_1}; \mathcal{C}_0} = \mathcal{R}^{S}(Z(\lambda_l, \dots, \lambda_1, \pi_u)) = \mathcal{L}_{S}(Z(\lambda_l, \dots, \lambda_1, \pi_u)).
$$

Since $\mathcal{C}_0$ is always singular, in order to obtain the Dirac index of $\mathcal{L}_S(Z(\lambda_l, \dots, \lambda_1, \pi_u))$ from that of $\mathcal{L}_S(Z(0, \dots, 0, \pi_u))$,
one needs to modify \cite[Theorem 4.7]{MPV} to the starting infinitesimal character $(\mathcal{C}_l^0; \dots; \mathcal{C}_1^0; \mathcal{C}_0)$ which is \emph{singular}.

Indeed, by Lemma \ref{lemma-translation}, \emph{Step 1} of \cite[Theorem 4.7]{MPV} remains true for all $X_{\mu}$ such that $\mu$ is of the form
\begin{equation} \label{eq-x}
\mu = (\mu_1, \mu_2, \dots, \mu_x; \mathcal{K});\quad \mu_1 \geq \mu_2 \geq \dots \geq \mu_x,
\end{equation}
i.e., $\mu$ is in the same dominant Weyl chamber as $X_{\mathcal{C}_l^{0}; \dots; \mathcal{C}_1^{0}; \mathcal{C}_0} := \mathcal{L}_{S}(Z(0, \dots, 0, \pi_u))$. In particular, this includes all
$X_{\mathcal{C}_l^{\lambda_l}; \dots; \mathcal{C}_1^{\lambda_1}; \mathcal{C}_0} = \mathcal{L}_{S}(Z(\lambda_l, \dots, \lambda_1, \pi_u))$
in the weakly good range. Therefore, the proposition holds for all $\mathcal{L}_{S}(Z(\lambda_l, \dots, \lambda_1, \pi_u))$ in the weakly good range.

As for $X_{\mathcal{C}_l^{\lambda_l}; \dots; \mathcal{C}_1^{\lambda_1}; \mathcal{C}_0}$ in weakly fair range but outside the weakly good range, the parameter $(\mathcal{C}_l^{\lambda_l}; \dots; \mathcal{C}_1^{\lambda_1}; \mathcal{C}_0)$ is no longer dominant, i.e., it is not in the same Weyl chamber as $(\mathcal{C}_l^{0}; \dots; \mathcal{C}_1^{0}; \mathcal{C}_0)$. So we need to study the change of Dirac index by crossing the wall between two chambers as in \emph{Step 2 -- 3} of \cite[Theorem 4.7]{MPV}.

Let $w \in W$ be the shortest Weyl group element mapping $\mu_{dom}$ in the dominant Weyl chamber to $w \mu_{dom} = (\mathcal{C}_l^{\lambda_l}; \dots; \mathcal{C}_1^{\lambda_1}; \mathcal{C}_0)$. Then $w$ is a permutation of the first $x$ coordinates of $\mu_{dom}$, where $x$ is given in \eqref{eq-x}. Since the first $x$ coordinates of $\mu_{dom}$ and $(\mathcal{C}_l^{\lambda_l}; \dots; \mathcal{C}_1^{\lambda_1}; \mathcal{C}_0)$ are strictly greater than the coordinates of $\mathcal{K}$, the same arguments in \emph{Step 2 -- 3} of \cite[Theorem 4.7]{MPV} remain valid, and the result follows.
\end{proof}


\begin{example}
We continue with Examples \ref{eg-so14-1} and \ref{eg-so14-2}.
The module $\mathcal{L}_S(Z(-2,-1,\pi_u))$ has chains $\mathcal{C}_2^{-2} = (4,3)^{1,1}$, $\mathcal{C}_1^{-1} = (3,2)^{1,1}$, $\mathcal{C}_0 = (1,1,0)$. By Proposition \ref{prop-translation}, it has Dirac index
\begin{align*} \mathrm{DI}(\mathcal{L}_S(Z(-2,-1,\pi_u))) = &- \widetilde{E}_{(4,3|1,0,-1|-2,-3)} + \widetilde{E}_{(3,3|1,0,-1|-2,-4)} \\ &+ \widetilde{E}_{(4,2|1,0,-1|-3,-3)} -\widetilde{E}_{(3,2|1,0,-1|-3,-4)}\\
= &- \widetilde{E}_{(4,3|1,0,-1|-2,-3)} -\widetilde{E}_{(3,2|1,0,-1|-3,-4)}.
\end{align*}
\hfill\qed
\end{example}

\begin{example}\label{eg-sp22-2}
We continue with Example \ref{eg-sp22-1}. By Theorem \ref{thm-DI-Aqlambda},
$$
\mathrm{DI}(A_{\frq_A}(\mathbb{C}_{n-\frac{p+q+1}{2}})) = \widetilde{E}_{({\bf 11|10,9})} - \widetilde{E}_{({\bf 10|11,9})} + \widetilde{E}_{({\bf 9|11,10})}.
$$
Moreover, by Section \ref{sec-DI-Xpqepsiloneta}, the Dirac index of $X(4, 2; 0, 0)$ is equal to
\begin{align*}
&+\widetilde{E}_{(5,4,3|2,1,0,-1,-2)}   -\widetilde{E}_{(3|2,1,0,-1,-2|-4,-5)}  -2\widetilde{E}_{(5,3|2,1,0,-1,-2|-4)}  + 2\widetilde{E}_{(4,3|2,1,0,-1,-2|-5)}\\
&+2\widetilde{E}_{(5,4|2,1,0,-1,-2|-3)}   -2\widetilde{E}_{(2,1,0,-1,-2|-3,-4,-5)}  - \widetilde{E}_{(5|2,1,0,-1,-2|-3,-4)}  + \widetilde{E}_{(4|2,1,0,-1,-2|-3,-5)}.
\end{align*}
Therefore, by Proposition \ref{prop-translation}, the Dirac index of $\mathcal{L}_S(Z(-6,X(4, 2; 0, 0)))$ is given by
\begin{align*}
&(-1)(-2)\widetilde{E}_{({\bf 4},5,3|2,1,0,-1,-2|-4,{\bf -3,-5})}  + (+1)(+2) \widetilde{E}_{({\bf 5},4,3|2,1,0,-1,-2|-5,{\bf -3,-4})}+ \\
& (+2)(+1)\widetilde{E}_{({\bf 3},5,4|2,1,0,-1,-2|-3,{\bf -4,-5})}  = +6\widetilde{E}_{(5,4,3|2,1,0,-1,-2|-3,-4,-5)}.
\end{align*} \hfill \qed
\end{example}

Based on Proposition \ref{prop-translation}, we will give explicit formulas on the multiplicities of $\widetilde{K}$-types
in $\mathrm{DI}(\mathcal{L}_S(Z(\lambda_l, \dots, \lambda_1, \pi_u)))$ in the following subsections.

\subsection{Dirac index for $\mathcal{L}_S(Z(\lambda_l, \dots, \lambda_1,X(r,s;\epsilon,\eta)))$}
We begin by studying the signs of different $\widetilde{K}$-types in $\mathrm{DI}(X(r,s;\epsilon,\eta))$:
\begin{lemma} \label{lem-sp1}
Consider the unipotent representation $X(r,s;\epsilon,\eta)$ of $Sp(2n, \bbR)$. Let  $\tau = w\Lambda_k - \rho_c$, $\tau' = w'\Lambda_k - \rho_c$ be two \emph{special} $\widetilde{K}$-types such that
\begin{align*}
w\Lambda_k &= (i_1, \dots, i_u, k, k-1, \dots, -k+1, -j_v, \dots, -j_1),\\
w^*\Lambda_k &= (i_1^*, \dots, i_u^*, k, k-1, \dots, -k+1, -j_v^*, \dots, -j_1^*),
\end{align*}
where
$$
\{i_1, \dots, i_u, j_1, \dots, j_v\} = \{i_1^*, \dots, i_u^*, j_1^*, \dots, j_v^*\} = \{k+1, \dots, n-k-1, n-k\}.
$$
Let $\xi \in S_{n-2k}$ denote the permutation
$$(i_1, \dots, i_u | j_1, \dots, j_v) \mapsto (i_1^*, \dots, i_u^* | j_1^*, \dots, j_v^*).$$
Then the sign of $E_{\tau}$ and the sign of $E_{\tau^*}$ in $\mathrm{DI}(X(r,s;\epsilon,\eta))$ differ by $\mathrm{det}(\xi)$.
\end{lemma}
\begin{proof}
Recall that the sign of $E_{\tau}$ in $\mathrm{DI}(X(r,s;\epsilon,\eta))$ is $(-1)^{M(\tau)}$ with $M(\tau)$ being given in \eqref{M-tau-special}.
Note that $j_1^\prime + \cdots + j_v^\prime - v$ is equal to the length of the following permutation
$$
\zeta: (n-k, n-k-1, \dots | \dots, k+2, k+1) \mapsto (i_1, \dots, i_u | j_1, \dots, j_v).
$$
Similar things apply to $E_{\tau^*}$ once we put
$$
\zeta^*: (n-k, n-k-1, \dots | \dots, k+2, k+1) \mapsto (i_1^*, \dots, i_u^* | j_1^*, \dots, j_v^*).
$$
Since $\xi = \zeta^*  \zeta^{-1} \in S_{n-2k}$, we have that
\begin{align*}
l(\xi) &\equiv l(\zeta^*) + l(\zeta^{-1})\\
&\equiv (j_1^*)^\prime + \cdots + (j_v^*)^\prime - v + j_1^\prime + \cdots + j_v^\prime - v \quad (\mathrm{mod}\ 2)\\
&\equiv (j_1^*)^\prime + \cdots + (j_v^*)^\prime - j_1^\prime - \cdots - j_v^\prime \quad (\mathrm{mod}\ 2)\\
&= M(\tau^*) - M(\tau)\quad (\mathrm{mod}\ 2)
\end{align*}
Therefore, $(-1)^{M(\tau^*)} = (-1)^{M(\tau)}  \det(\xi)$.
\end{proof}
\begin{rmk}\label{rmk-lem-sp1}
The same result holds if $\tau$ and $\tau^*$ are both non-special.
\end{rmk}

Now let us compare the signs of $\mathrm{DI}(X(r,s;\epsilon,\eta))$ of a special $\tau$ and a non-special $\tau''$.

\begin{lemma} \label{lem-sp2}
Consider the unipotent representation $X(r,s;\epsilon,\eta)$ of $Sp(2n, \bbR)$.
Let $\tau = w\Lambda_k - \rho_c$ be \emph{special} with
\begin{align*}
w\Lambda_k = (i_1, \dots, i_u; k, k-1, \dots, -k+1; -j_{v} \dots, -j_1).
\end{align*}
Let $\tau'' = w''\Lambda_k - \rho_c$ be \emph{non-special} with
\begin{align*}
w''\Lambda_k = (i_1, \dots, i_u; k-1, \dots, -k+1, -k; -j_{v}, \dots, -j_1)
\end{align*}
where
$$
\{i_1, \dots, i_u, j_1, \dots, j_{v}\} = \{k+1, \dots, n-k-1, n-k\}.
$$
Then the signs of $E_{\tau}$ and $E_{\tau''}$ in $\mathrm{DI}(X(r,s;\epsilon,\eta))$ differ by $(-1)^r = (-1)^s$.
\end{lemma}
\begin{proof}
As before, the sign of $E_{\tau}$ in $\mathrm{DI}(X(r,s;\epsilon,\eta))$  is $(-1)^{M(\tau)}$.
On the other hand, as discussed in Section \ref{sec-DI-Xpqepsiloneta}, the sign of the non-special $E_{\tau''}$ is $(-1)^{N(-w_0^K\tau'')}$, where
$$N(-w_0^K\tau'') = i_1^\prime + \cdots + i_u^\prime + s(\eta + u) +\frac{s(s-1)}{2}+\frac{n(n+1)}{2}.$$
It remains to verify that $M(\tau) + N(-w_0^K\tau'')$ has the same parity as $r$ (and $s$). Indeed,
\begin{align*}
&M(\tau) + N(-w_0^K\tau'')\\
=\ &\frac{(n-2k)(n-2k+1)}{2} + r(\epsilon + v)+ s(\eta + u)+\frac{r(r-1)}{2}+\frac{s(s-1)}{2}+\frac{n(n+1)}{2} \\
\equiv\ &\frac{n(n+1)}{2} - k(2n+1) + 2k^2 + r(\epsilon + v)+ r(\eta + u)+\frac{r(r-1)}{2}+\frac{s(s-1)}{2}+\frac{n(n+1)}{2} \\
=\ &k + r(\epsilon + \eta + n -2k)+\frac{r(r-1)}{2}+\frac{s(s-1)}{2}\\
\equiv\ &k + r(\epsilon + \eta + n)+\frac{r(r-1)}{2}+\frac{s(s-1)}{2}\\
\equiv\ & \begin{cases} k+0+\frac{r}{2}+\frac{s}{2} = k+k \equiv 0\ (\mathrm{mod}\ 2) & \text{if}\ r, s\ \text{are even}\\
k + \epsilon + \eta + n +\frac{r-1}{2}+\frac{s-1}{2} = k + (\epsilon + \eta + n) + k - 1 \equiv 1\ (\mathrm{mod}\ 2) & \text{if}\ r, s\ \text{are odd}\end{cases}
\end{align*}
The last step need a bit explanation. Assume that $r$ and $s$ are both odd. Then Theorem 3.7(2) of \cite{BP15} says that $E_{\tau}$ and $E_{-w_0^K\tau''}$ appear in $H_D(r,s;\epsilon,\eta)$
if and only if $\epsilon + \eta \equiv n \ ({\rm mod} 2)$. That is, if and only if $\epsilon + \eta + n$ is even.
\end{proof}

Combining the above two lemmas, we have the following result.

\begin{prop} \label{prop-paritysp}
Let $\tau = w\Lambda_k - \rho_c$ and $\tau^* = w^*\Lambda_k - \rho_c$ be such that
\begin{align*}
w\Lambda_k &= (i_1, \dots, i_u, i_{u+1}; k-1, \dots, -(k-1); -j_v, \dots, -j_1);\\
w^*\Lambda_k &= (i_1^*, \dots, i_u^*, i_{u+1}^*; k-1, \dots, -(k-1); -j_v^*, \dots, -j_1^*),
\end{align*}
where
$$
\{i_1, \dots, i_u, i_{u+1}, j_1, \dots, j_v\} = \{i_1^*, \dots, i_u^*, i_{u+1}^*, j_1^*, \dots, j_v^*\} = \{k, k+1, \dots, n-k\}.
$$
Then the signs of $E_{\tau}$ and $E_{\tau^*}$ in $\mathrm{DI}(X(r,s; \epsilon, \eta))$
differ by $\mathrm{det}(\xi)$, where $\xi \in S_{n-2k+1}$ is the permutation defined by
$$(i_1, \dots, i_u, i_{u+1}| j_1, \dots, j_v) \mapsto (i_1^*, \dots, i_u^*, i_{u+1}^*| j_1^*, \dots, j_v^*).$$
\end{prop}
\begin{proof}
By Lemma \ref{lem-sp1} and Remark \ref{rmk-lem-sp1}, the proposition holds if $i_{u+1} = i_{u+1}^* = k$ (that is, if $\tau$ and $\tau^*$ are both special) and
$j_{v} = j_{v}^* = k$ (that is, if $\tau$ and $\tau^*$ are both non-special). It remains to study the case when $\tau$ is special while $\tau^*$ is non-special.
Thus we assume that $i_{u+1} = k$ and $j_v^* = k$. In particular, it suffices to check that the statement holds for
\begin{align*}
w\Lambda_k &= &&(i_1, \dots, i_u, &&k; &&k-1, \dots, -k+1; &&-(k+1), &&-j_{v-1}, \dots, -j_1);\\
w^*\Lambda_k &= &&(i_1, \dots, i_u, &&k+1; &&k-1, \dots, -k+1; &&-k, &&-j_{v-1}, \dots, -j_1).
\end{align*}
Let $\tau^{\sharp} = w^{\sharp}\Lambda_k - \rho_c$ be given by
$$
w^{\sharp}\Lambda_k = (i_1, \dots, i_u, k+1, k;\ \ k-1, \dots, -k+1;\ \  -j_{v-1}, \dots, -j_1).
$$
By Lemma \ref{lem-sp2}, in $\mathrm{DI}(X(r,s;\epsilon,\eta))$, we have
$$
{\rm sign} (\widetilde{E}_{\tau^{\sharp}})\
{\rm sign}(\widetilde{E}_{\tau^*})=(-1)^r.
$$
Since
$$
M(\tau^{\sharp}) = j_1' + \dots + j_{v-1}' + r(\epsilon+v-1) + \frac{r(r+1)}{2},
$$
while
$$
M(\tau) = j_1' + \dots + j_{v-1}' + 1 + r(\epsilon+v) + \frac{r(r+1)}{2} = M(\tau^{\sharp}) + r + 1,
$$
we have that
$$
{\rm sign} (\widetilde{E}_{\tau})\
{\rm sign}(\widetilde{E}_{\tau^{\sharp}})=(-1)^{r+1}.
$$
Therefore,
$$
{\rm sign} (\widetilde{E}_{\tau}) \
{\rm sign} (\widetilde{E}_{\tau^*})=(-1)^r (-1)^{r+1} = -1.
$$
On the other hand, the permutation $\xi$ in $S_{n-2k+1}$
moving $w\Lambda_k$ to $w^*\Lambda_k$ is
$$(i_1, \dots, i_u, k, j_{1}, \dots, j_{v-1}, k+1) \mapsto (i_1, \dots, i_u, k+1, j_{1}, \dots, j_{v-1},k)$$
which has determinant $-1$. Hence the result follows.
\end{proof}

From now on, given a $(\frg, K)$ module $\pi$ and a $\widetilde{K}$-type $E_{\tau}$, we will use $|\tau: {\rm DI}(\pi)|$ (resp., $[\tau: {\rm DI}(\pi)]$) for the multiplicity (resp., \emph{signed} multiplicity) of $E_{\tau}$ in the Dirac index of $\pi$.


\begin{thm} \label{thm-mainsp}
Let $G = Sp(2n,\mathbb{R})$. Consider $\mathcal{L}_S(Z(\lambda_l,\dots,\lambda_1,X(r,s;\epsilon,\eta)))$ such that its chains are of the form \eqref{eq-interlaced}.
Let $\mathcal{A}^+ \amalg \mathcal{A}^-$ (resp., $\mathcal{B}_t^+ \amalg \mathcal{B}_t^-$) be any partition of $\caA$ (resp., $\caB_t$). Then the multiplicity
$$
\left|
\widetilde{E}_{(\mathcal{A}^+; \mathcal{R}_l; \mathcal{B}_l^+; \dots; \mathcal{R}_1; \mathcal{B}_1^+;\
\mathcal{K};\ \underline{\mathcal{A}^-; \mathcal{R}_l; \mathcal{B}_l^-; \dots; \mathcal{R}_1; \mathcal{B}_1^-})}:\mathrm{DI}(\mathcal{L}_S(Z(\lambda_l,\dots,\lambda_1,X(r,s;\epsilon,\eta))))
\right|
$$
(where $(\underline{j_1, \dots, j_v}) := (-j_v, \dots, -j_1)$, $\caK= (k-1, \dots, 1,0,-1,\dots, -k+1)$) is equal to
$$
\displaystyle \sum_{\substack{\mathcal{M}_t\, \amalg\, \mathcal{N}_t = \mathcal{R}_t;\\ |\mathcal{M}_l| = q_l - |\mathcal{A}^-|,\
|\mathcal{N}_l| = p_l - |\mathcal{A}^+|;\\
|\mathcal{M}_t| = q_t,\ |\mathcal{N}_t| = p_t\ \text{for}\ t<l;}}
\left|\widetilde{E}_{(\mathcal{M}_l;\mathcal{B}_l^+; \mathcal{M}_{l-1}; \dots; \mathcal{M}_1;\mathcal{B}_1^+;\ \mathcal{K};\ \underline{\mathcal{N}_l;\mathcal{B}_l^-;\mathcal{N}_{l-1}; \dots; \mathcal{N}_1;\mathcal{B}_1^-})}: \mathrm{DI}(X(r,s;\epsilon,\eta))\right|
$$
if $p_l \geq |\mathcal{A}^+|$ and $q_l \geq |\mathcal{A}^-|$. Otherwise, the multiplicity is equal to zero.
\end{thm}
\begin{proof}
We will only prove the theorem when $l =1$, and the argument for multiple chains is similar.

There are two possibilities for $\mathcal{C}_1^{\lambda}$:
\begin{center}
(a)\ \begin{tabular}{ccccc}
$\mathcal{C}_1^{\lambda} = (\mathcal{A}$ & $\mathcal{R}_1)^{p_1, q_1}$ & &  \tabularnewline
 & $(\mathcal{R}_1;$ \quad \quad \quad & $\mathcal{B}_1$; & $\mathcal{K}) = \mathcal{C}_0$
\end{tabular}\ or\quad  (b)\ \begin{tabular}{ccccc}
 & $\mathcal{C}_1^{\lambda} = (\mathcal{R}_{1})^{p_{1}, q_{1}}$ &  &  &   \tabularnewline
 $(\mathcal{A};$ & $\mathcal{R}_{1}$;   & $\mathcal{B}_1$; & $\mathcal{K}) = \mathcal{C}_0$
\end{tabular}
\end{center}


We will only prove case (a). Case (b) is similar. In order to obtain the $\widetilde{K}$-type
$$\widetilde{E}_{(\mathcal{A}^+; \mathcal{R}_1; \mathcal{B}_1^+;\ \mathcal{K};\ \underline{\mathcal{A}^-; \mathcal{R}_1; \mathcal{B}_1^-})}$$
in the Dirac index of $\mathcal{L}_S(Z(\lambda, X(r,s;\epsilon,\eta)))$ by Proposition \ref{prop-translation}, one needs to take
\begin{equation}\label{eq-two-components}
\begin{aligned}
&a_{\mathcal{M}, \mathcal{N}} \widetilde{E}_{(\mathcal{A}^+;\mathcal{N}| \mathcal{A}^-;\mathcal{M}) + (\lambda, \dots, \lambda\ |\ -\lambda, \dots, -\lambda)} &&\text{in}\quad \mathrm{DI}(A_{\frq_A}(\mathbb{C}_{n - \frac{p_1+q_1+1}{2}}))\\
&b_{\mathcal{M},\mathcal{N}} \widetilde{E}_{(\mathcal{M};\mathcal{B}_1^+;\ \mathcal{K};\ \underline{\mathcal{N};\mathcal{B}_1^-})} &&\text{in}\quad \mathrm{DI}(X(r,s;\epsilon,\eta))
\end{aligned}
\end{equation}
for some $\mathcal{M} \, \amalg \, \mathcal{N} = \mathcal{R}_1$, where $a_{\mathcal{M},\mathcal{N}}$ and  $b_{\mathcal{M},\mathcal{N}}$
are the corresponding signed multiplicities. In particular, $a_{\mathcal{M},\mathcal{N}} = 1$ or $-1$. By \eqref{eq-two-components}, we must have
$$|\mathcal{N}| = p_1 - |\mathcal{A}^+| \geq 0, \quad |\mathcal{M}| = q_1 - |\mathcal{A}^-| \geq 0$$
as stated in the last part of the theorem.

Fix any choice of $\mathcal{M} \, \amalg \, \mathcal{N} = \mathcal{R}_1$.
Put
$$
N:=|\mathcal{N}| = p_1 - |\mathcal{A}^+|,
\quad M:=|\mathcal{M}| = q_1 - |\mathcal{A}^-|.
$$
By Proposition \ref{prop-translation}, for each possibility of the $\widetilde{K}$-types
of the form \eqref{eq-two-components}, they contribute to $\mathrm{DI}(\mathcal{L}_S(Z(\lambda,X(r,s;\epsilon,\eta))))$ with signed multiplicity:
\begin{align*}
&a_{\mathcal{M},\mathcal{N}}b_{\mathcal{M},\mathcal{N}}  \widetilde{E}_{(\mathcal{A}^+;\mathcal{N};\mathcal{M};\mathcal{B}_1^+;\ \mathcal{K};\ \underline{\mathcal{A}^-;\mathcal{M};\mathcal{N};\mathcal{B}_1^-})}\\ =\ &(-1)^{MN}a_{\mathcal{M},\mathcal{N}}b_{\mathcal{M},\mathcal{N}} \widetilde{E}_{(\mathcal{A}^+;\mathcal{M};\mathcal{N};\mathcal{B}_1^+;\ \mathcal{K};\ \underline{\mathcal{A}^-;\mathcal{M};\mathcal{N};\mathcal{B}_1^-})}  \\
=\ &(-1)^{MN}a_{\mathcal{M},\mathcal{N}}b_{\mathcal{M},\mathcal{N}}  \widetilde{E}_{(\mathcal{A}^+;\mathcal{R}_1;\mathcal{B}_1^+;\ \mathcal{K};\ \underline{\mathcal{A}^-;\mathcal{R}_1;\mathcal{B}_1^-})}.
\end{align*}
We \emph{claim} that  $a_{\mathcal{M},\mathcal{N}}b_{\mathcal{M},\mathcal{N}}$ has a constant sign.
Indeed, let
$$
\mathcal{M} \,\amalg\, \mathcal{N} = \mathcal{M}^* \,\amalg \,\mathcal{N}^* = \mathcal{R}_1, \quad
|\mathcal{M}|=|\mathcal{M}^*|= M, \quad |\mathcal{N}| = |\mathcal{N}^*| = N.
$$
By the knowledge of Dirac index from \cite{HKP} for one-dimensional modules,
$$
a_{\mathcal{M}^*,\mathcal{N}^*} = \det(\xi) a_{\mathcal{M},\mathcal{N}},\quad \text{where}\ \ \xi:(\mathcal{A}^+;\mathcal{N}| \mathcal{A}^-;\mathcal{M}) \mapsto (\mathcal{A}^+;\mathcal{N}^*| \mathcal{A}^-;\mathcal{M}^*).
$$
On the other hand,
by Proposition \ref{prop-paritysp},
$$b_{\mathcal{M}^*,\mathcal{N}^*} = \det(\zeta) b_{\mathcal{M},\mathcal{N}},\quad \text{where}\ \ \zeta:(\mathcal{M};\mathcal{B}_1^+|\mathcal{N};\mathcal{B}_1^-) \mapsto (\mathcal{M}^*;\mathcal{B}_1^+|\mathcal{N}^*;\mathcal{B}_1^-)$$
It is obvious  that $\det(\xi)\det(\zeta) =1$. Thus the claim holds, and the desired multiplicity is
\begin{align*}
\sum_{\substack{\mathcal{M} \, \amalg \, \mathcal{N} = \mathcal{R}_1;\\ |\mathcal{M}| = M,\
|\mathcal{N}| = N}} |a_{\mathcal{M},\mathcal{N}}| |b_{\mathcal{M},\mathcal{N}}|
 =\ \sum_{\substack{\mathcal{M} \, \amalg \, \mathcal{N} = \mathcal{R}_1;
 \\ |\mathcal{M}| =M,\
|\mathcal{N}| = N}} |b_{\mathcal{M},\mathcal{N}}|,
\end{align*}
which finishes the proof.
\end{proof}

\begin{cor} \label{cor-mainsp}
Retain the setting in Theorem \ref{thm-mainsp}. The multiplicity
$$
\left|
\widetilde{E}_{(\mathcal{A}^+; \mathcal{R}_l; \mathcal{B}_l^+; \dots; \mathcal{R}_1; \mathcal{B}_1^+;\
\mathcal{K};\ \underline{\mathcal{A}^-; \mathcal{R}_l; \mathcal{B}_l^-; \dots; \mathcal{R}_1; \mathcal{B}_1^-})}:\mathrm{DI}(\mathcal{L}_S(Z(\lambda_l,\dots,\lambda_1,X(r,s;\epsilon,\eta))))
\right|$$
is equal to
$$
{|\mathcal{R}_l| \choose p_l - |\mathcal{A}^+|}
\prod_{t=1}^{l-1} { p_t + q_t \choose \ p_t }
\left|\widetilde{E}_{(m-k, m-k-1, \dots, m-k-w+1;\ \mathcal{K};\ \underline{m-k-w,\dots, k+1, k})}: \mathrm{DI}(X(r,s;\epsilon,\eta))\right|,
$$
where $w := \sum_{t=1}^{l} (q_t + |\mathcal{B}_t^+|) - |\mathcal{A}^-|$.
\end{cor}
\begin{proof}
By Proposition 3.4 and Theorem 3.7 of \cite{BP15}, the multiplicities in each summand of Theorem
\ref{thm-mainsp} are all equal. So the result follows from counting the number of terms in the summation.
\end{proof}

\begin{example}\label{eg-sp22-3}
We continue with the module $\mathcal{L}_S(Z(-6,X(4,2, 0,0)))$ in Example \ref{eg-sp22-2},
whose chains are of the form
\begin{align*} \mathcal{C}_1^{-6} = (5,4,3)^{1,2} \quad \mathcal{C}_0 = (5,4,3,2,1,0,-1,-2).
\end{align*}
Thus
$$\mathcal{R}_1 = \{5,4,3\}, \quad \mathcal{K}= (2,1,0,-1,-2), \quad \mathcal{A}=\mathcal{B}=\emptyset.$$
By Corollary \ref{cor-mainsp},
the Dirac index of $\mathcal{L}_S(Z(-6,X(4,2;0,0)))$ consists of the single $\widetilde{K}$-type
$$\widetilde{E}_{(\mathcal{R}_1;2,1,0,-1,-2;\underline{\mathcal{R}_1})} = \widetilde{E}_{(5,4,3;2,1,0,-1,-2;-3,-4,-5)} = E_{(0,0,0,0,0,0,0,0,0,0,0)}$$
with $w = (q_1 + |\mathcal{B}_1^+|) - |\mathcal{A}^-| = 2 + 0 - 0 = 2$. Hence its multiplicity is equal to
$$
{3 \choose 1 - 0} \left|\widetilde{E}_{(5,4;2,1,0,-1,-2;-3)} : {\rm DI}(X(4, 2, 0,0))\right| = 3\times 2 = 6.  
$$
\hfill\qed
\end{example}

\subsection{Dirac index for $\mathcal{L}_S(Z(\lambda_l, \dots, \lambda_1,X(r,s)))$}
Now we set $G = SO^*(2n)$ and $\pi_u = X(r,s)$.
To obtain an result on
the Dirac index of $\mathcal{L}_S(Z(\lambda_l, \dots, \lambda_1, X(r,s)))$ analogous to Theorem \ref{thm-mainsp},
one needs to know the signs of $\widetilde{E}_{\tau}$ in $\mathrm{DI}(X(r,s))$
(c.f. Proposition \ref{prop-paritysp}). Indeed, \eqref{sign-tau-Xpq} gives the sign of
each $\widetilde{K}$-type appearing in $\mathrm{DI}(X(r,s))$, which immediately implies the following.

\begin{prop} \label{prop-parityso}
Let $\tau = w\Lambda_k - \rho_c$, $\tau^* = w^*\Lambda_k - \rho_c$ such that
\begin{align*}
w\Lambda_k &= (i_1, \dots, i_u; k, \dots, -k; -j_v, \dots, -j_1),\\
w^*\Lambda_k &= (i_1^*, \dots, i_u^*; k, \dots, -k; -j_v^*, \dots, -j_1^*),
\end{align*}
where
$$
\{i_1, \dots, i_u, j_1, \dots, j_v\} = \{i_1^*, \dots, i_u^*, j_1^*, \dots, j_v^*\} = \{k+1, \dots, n-k-1\}.
$$
Then the sign of $E_{\tau}$ and $E_{\tau^*}$ appearing in $\mathrm{DI}(X(r,s))$
differ by $\mathrm{det}(\xi)$, where $\xi \in S_{n-2k-1}$ is the permutation defined by
$$(i_1, \dots, i_u| j_1, \dots, j_v) \mapsto (i_1^*, \dots, i_u^*| j_1^*, \dots, j_v^*).$$
\end{prop}


Now we can state the following result for $\mathrm{DI}(\mathcal{L}_S(Z(\lambda_l, \dots, \lambda_1, X(r,s))))$.

\begin{thm} \label{thm-mainso}
Let $G$ be $SO^*(2n)$. Consider $\mathcal{L}_S(Z(\lambda_l, \dots, \lambda_1, X(r,s)))$ such that its chains are of the form \eqref{eq-interlaced}.
Let $\mathcal{A}^+ \amalg \mathcal{A}^-$ (resp., $\mathcal{B}_t^+ \amalg \mathcal{B}_t^-$) be any partition of $\caA$ (resp., $\caB_t$). Then the multiplicity
$$
\left|
\widetilde{E}_{(\mathcal{A}^+; \mathcal{R}_l; \mathcal{B}_l^+; \dots; \mathcal{R}_1; \mathcal{B}_1^+;\
\mathcal{K};\ \underline{\mathcal{A}^-; \mathcal{R}_l; \mathcal{B}_l^-; \dots; \mathcal{R}_1; \mathcal{B}_1^-})}:\mathrm{DI}(\mathcal{L}_S(Z(\lambda_l,\dots,\lambda_1,X(r,s))))
\right|
$$
(here $\caK= (k, \dots, 1, 0,-1,\dots, -k)$) is equal to
$$
{r+s\choose r}  {|\mathcal{R}_l| \choose p_l - |\mathcal{A}^+|}
\prod_{t=1}^{l-1} { p_t + q_t \choose \ p_t }
$$
if $p_l \geq |\mathcal{A}^+|$ and $q_l \geq |\mathcal{A}^-|$. Otherwise, the multiplicity is equal to zero.
\end{thm}
\begin{proof}
The proof is analogous to that of Theorem \ref{thm-mainsp} and Corollary \ref{cor-mainsp}. Namely, the multiplicity is equal to
$$
\displaystyle \sum_{\substack{\mathcal{M}_t\, \amalg\, \mathcal{N}_t = \mathcal{R}_t;\\ |\mathcal{M}_l| = q_l - |\mathcal{A}^-|,\
|\mathcal{N}_l| = p_l - |\mathcal{A}^+|;\\
|\mathcal{M}_t| = q_t,\ |\mathcal{N}_t| = p_t\ \text{for}\ t<l;}}
\left|\widetilde{E}_{(\mathcal{M}_l;\mathcal{B}_l^+; \mathcal{M}_{l-1}; \dots; \mathcal{M}_1;\mathcal{B}_1^+;\ \mathcal{K};\ \underline{\mathcal{N}_l;\mathcal{B}_l^-;\mathcal{N}_{l-1}; \dots; \mathcal{N}_1;\mathcal{B}_1^-})}: \mathrm{DI}(X(r,s))\right|
$$
By Theorem \ref{thm-Dirac-Xpq}, all the multiplicities of the above formulas
are equal to ${r+s\choose r}$. Hence the result follows by counting the
amount of summands.
\end{proof}

\begin{remark}
An analogous statement for Corollary \ref{cor-mainsp} and Theorem \ref{thm-mainso} holds also when $\pi_u = \mathrm{triv}$ is the trivial representation of $Sp(2m,\mathbb{R})$ or $SO^*(2m)$.
Under this setting,
$$\mathcal{C}_0 = \begin{cases} (m, \dots, 2,1) &\text{for}\ G = Sp(2m,\mathbb{R});\\
(m-1, \dots, 1,0) &\text{for}\ G = SO^*(2m)
\end{cases}$$
and $\mathcal{L}_S(Z(\lambda_l, \dots, \lambda_1,\mathrm{triv}))$ is a weakly fair $\aq(\lambda)$-module since $\mathrm{triv}$ is one-dimensional, and hence Theorem \ref{thm-DI-Aqlambda} applies.
Note that the infinitesimal character of $\mathcal{L}_S(Z(\lambda_l, \dots, \lambda_1,\mathrm{triv}))$
satisfies Theorem \ref{thm-HP} if its chains satisfy
\begin{equation} \label{eq-interlaced3}
\begin{tabular}{ccccccccccccc}
$(\mathcal{A};$ & $\mathcal{R}_l)^{p_l,q_l}$ &  & $(\mathcal{R}_{l-1})^{p_{l-1},q_{l-1}}$ & $\dots$  & $\dots$ & $(\mathcal{R}_{1})^{p_{1},q_{1}}$ &  &  &  &  &  & \tabularnewline
 & $(\mathcal{R}_l;$ \quad \quad \quad & $\mathcal{B}_l$; & $\mathcal{R}_{l-1}$;\quad \quad \quad \quad  & $\mathcal{B}_{l-1}$; & $\dots$; & $\mathcal{R}_{1}$; \quad \quad & $\mathcal{B}_1) = \mathcal{C}_0$, &  &  &  &
\end{tabular},
\end{equation}
or when $G = Sp(2m,\mathbb{R})$ and $0 \in \mathcal{C}_1$, we may have
\begin{equation} \label{eq-interlaced4}
\begin{tabular}{ccccccccccccc}
$(\mathcal{A};$ & $\mathcal{R}_l)^{p_l,q_l}$ &  & $(\mathcal{R}_{l-1})^{p_{l-1},q_{l-1}}$ & $\dots$  & $\dots$ & $(\mathcal{R}_{1};$ & $0)^{p_{1},q_{1}}$ & \tabularnewline
 & $(\mathcal{R}_l;$ \quad \quad \quad & $\mathcal{B}_l$; & $\mathcal{R}_{l-1}$;\quad \quad \quad \quad  & $\mathcal{B}_{l-1}$; & $\dots$; & \quad \quad \quad $\mathcal{R}_{1}) = \mathcal{C}_0$,
\end{tabular},
\end{equation}
or
\begin{equation} \label{eq-interlaced5}
\begin{tabular}{ccccccccccccc}
$\mathcal{C}_1 = (\mathcal{A};$ & $\mathcal{R}_1;$ & $0)^{p_{1},q_{1}}$ & \tabularnewline
 &\quad \quad $(\mathcal{R}_1) = \mathcal{C}_0$,
\end{tabular}.
\end{equation}
It turns out that in all cases, the multiplicity
\begin{align*}
&\left|
\widetilde{E}_{(\mathcal{A}^+; \mathcal{R}_l; \mathcal{B}_l^+; \dots; \mathcal{R}_1; \mathcal{B}_1^+;\ \underline{\mathcal{A}^-; \mathcal{R}_l; \mathcal{B}_l^-; \dots; \mathcal{R}_1; \mathcal{B}_1^-})}:\mathrm{DI}(\mathcal{L}_S(Z(\lambda_l,\dots,\lambda_1,\mathrm{triv})))
\right|\ \text{in}\ \eqref{eq-interlaced3},\ \text{or} \\
&\left|
\widetilde{E}_{(\mathcal{A}^+; \mathcal{R}_l; \mathcal{B}_l^+; \dots; \mathcal{R}_1;\ 0; \ \underline{\mathcal{A}^-; \mathcal{R}_l; \mathcal{B}_l^-; \dots; \mathcal{R}_1})}:\mathrm{DI}(\mathcal{L}_S(Z(\lambda_l,\dots,\lambda_1,\mathrm{triv})))
\right|\ \text{in}\ \eqref{eq-interlaced4},\ \text{or} \\
&\left|
\widetilde{E}_{(\mathcal{A}^+; \mathcal{R}_1;\ 0; \ \underline{\mathcal{A}^-; \mathcal{R}_l})}:\mathrm{DI}(\mathcal{L}_S(Z(\lambda_l,\dots,\lambda_1,\mathrm{triv})))
\right|\ \text{in}\ \eqref{eq-interlaced5}
\end{align*}
are equal to
$${|\mathcal{R}_l| \choose p_l - |\mathcal{A}^+|}  \prod_{t=1}^{l-1} { p_t + q_t \choose \ p_t }$$
(in the third case, we have $l = 1$ and the second term above will not show up).
\end{remark}

\section{Parity of spin-lowest $K$-types and cancellation in the Dirac index}   \label{sec-parity}

Let $\pi$ be an irreducible unitary $(\frg, K)$ module. Fix a $K$-type $\mu_0$ of $\pi$. Let $\mu$ be any $K$-type of $\pi$. We assume that
\begin{equation}\label{integral-pi}
\langle \mu - \mu_0, \zeta \rangle \in \bbZ
\end{equation}
We call the parity of this integer the \emph{parity of $\mu$}, and denote it by $p(\mu)$.

\begin{defi}
We say that the $K$-type $E_{\mu}$ of $\pi$ is \emph{related to} the $\widetilde{K}$-type $E_{\nu}$ of $H_D(\pi)$ if $E_{\nu}$ is a PRV-component \cite{PRV} of $E_{\mu}\otimes {\rm Spin}_G$.
\end{defi}

The following result aims to clarify the link between the possible cancellations of $\widetilde{K}$-types in $H_D(\pi)$ and the parities of the spin-lowest $K$-types of $\pi$.

\begin{thm}\label{thm-parity-cancellation}
Let $\pi$ be an irreducible unitary $(\frg, K)$ module such that \eqref{integral-pi} holds. Then ${\rm Hom}_{\widetilde{K}}(H_D^+(\pi), H_D^-(\pi))=0$ if and only if for each $\widetilde{K}$-type $E_{\nu}$ of $H_D(\pi)$ (if exists), all the spin lowest $K$-types of $\pi$ which are related to $E_{\nu}$ have the same parity.
\end{thm}
\begin{proof}
It suffices to consider the case that $H_D(\pi)\neq 0$. Then ${\rm Hom}_{\widetilde{K}}(H_D^+(\pi), H_D^-(\pi))=0$ if and only if the occurrences of $E_{\nu}$ either \emph{all} live in $H_D^+(\pi)$, or \emph{all} live in $H_D^-(\pi)$. Here $E_{\nu}$ runs over all the \emph{distinct} $\widetilde{K}$-components of $H_D(\pi)$.

Take two arbitrary $K$-types $E_{\mu_1}$ and $E_{\mu_2}$ of $\pi$ which are related to the $\widetilde{K}$-type $E_{\nu}$. Then there exist $w_1, w_2\in W(\frg, \frt)^1$ such that
$$
\{(w_1\rho -\rho_c)+ w_0^K \mu_1\}=\{(w_2\rho -\rho_c) + w_0^K \mu_2\}=\nu.
$$
Here $w_0^K$ stands for the unique longest element of $W(\frk, \frt)$. Removing the two brackets, we have
\begin{equation}\label{mu-gamma-1-2}
(w_1\rho -\rho_c)+ \mu_1- (\mu_1- w_0^K \mu_1)+\gamma_1=(w_2\rho -\rho_c)+\mu_2 -(\mu_2-w_0^K \mu_2) +\gamma_2,
\end{equation}
where each $\gamma_i$ is a non-negative integer combination of roots in $\Delta^+(\frk, \frt)$, and so is each $\mu_i- w_0^K \mu_i$.
Therefore,
$$
\mu_1- (\rho-w_1\rho)- (\mu_1- w_0^K \mu_1)+\gamma_1=\mu_2- (\rho-w_2\rho)-(\mu_2-w_0^K \mu_2)+\gamma_2.
$$
In other words,
$$
\mu_1- \langle\Phi_{w_1}\rangle- (\mu_1- w_0^K \mu_1)+\gamma_1=\mu_2- \langle\Phi_{w_2}\rangle-(\mu_2-w_0^K \mu_2)+\gamma_2.
$$
Taking inner products with $\zeta$ and passing to ${\rm mod}\ 2$, we have that
\begin{equation}\label{parity-delta-length-w}
p(\mu_1)+ l(w_1)\equiv p(\mu_2)+ l(w_2) \quad ({\rm mod}\ 2)
\end{equation}
by using \eqref{zeta-k-p}. Now the desired conclusion follows from \eqref{parity-delta-length-w} and Lemma \ref{lemma-even-odd-spinG}.
\end{proof}

\begin{example} Let us consider the following irreducible unitary representation $\pi$ of $Sp(10, \bbR)$.
\begin{verbatim}
G:Sp(10,R)
set p=parameter (KGB (G)[444],[4,2,4,0,1]/1,[1,0,2,-1,1]/1)
is_unitary(p)
Value: true
print_branch_irr_long (p,KGB (G,31), 65)
m  x    lambda               hw                      dim    height
1  179  [ 2, 2, 2, 2, 1 ]/1  [  2, -1, -1, -3, -3 ]  1200   42
1  32   [ 3, 2, 2, 0, 0 ]/1  [  3,  0, -1, -3, -3 ]  5400   50
1  4    [ 3, 2, 2, 0, 0 ]/1  [  2, -1, -2, -3, -4 ]  5120   51
1  91   [ 3, 3, 2, 0, 1 ]/1  [  2, -2, -2, -4, -4 ]  2250   57
1  180  [ 3, 3, 2, 2, 1 ]/1  [  4, -1, -1, -3, -3 ]  3850   58
1  179  [ 3, 3, 2, 2, 1 ]/1  [  3, -1, -1, -3, -4 ]  7425   58
1  179  [ 4, 2, 2, 2, 1 ]/1  [  2, -1, -1, -3, -5 ]  9240   60
1  56   [ 3, 3, 2, 1, 0 ]/1  [  3,  0, -2, -3, -4 ]  16170  60
1  1    [ 3, 3, 2, 1, 0 ]/1  [  4,  1, -1, -3, -3 ]  16170  61
1  4    [ 4, 2, 2, 1, 0 ]/1  [  2, -1, -3, -3, -5 ]  8624   63
1  4    [ 3, 3, 3, 0, 0 ]/1  [  3, -1, -2, -4, -4 ]  9625   63
\end{verbatim}
By Theorem \ref{thm-HP},  the \texttt{atlas} height of any spin-lowest $K$-type of $\pi$ is less than or equal to $55$. It turns out that the first three $K$-types are exactly all the spin lowest $K$-types of $\pi$:
$$
\mu_1=(2, -1, -1, -3, -3), \ \mu_2=(3,  0, -1, -3, -3),  \ \mu_3=(2, -1, -2, -3, -4).
$$
Recall that $\zeta=(\frac{1}{2}, \frac{1}{2}, \frac{1}{2}, \frac{1}{2}, \frac{1}{2})$. Thus $\mu_2$ and $\mu_3$ have the same parity, which is opposite to that of $\mu_1$.
Moreover, $\mu_1$ contributes $E_{(1,1,1,0,0)}$  to $H_D(\pi)$, while $\mu_2$ and $\mu_3$ both contribute $E_{(0,0,-1,-1,-1)}$. We conclude that there is no cancellation when passing from $H_D(\pi)$ to ${\rm DI}(\pi)$.

We now use results in the previous sections to compute $\mathrm{DI}(\pi)$. Indeed,
$\pi = \mathcal{L}_S(-2,\mathrm{triv})$ corresponds to the chains
$$\mathcal{C}_1^{-2} = (3,2,1)^{1,2}, \quad \quad \mathcal{C}_0 = (2,1)$$
in the form of \eqref{eq-interlaced3}. Therefore, $\mathcal{A} = \{\bf 3\}$, $\mathcal{R}_1 = \{2,1\}$, and
\begin{align*}
&|\widetilde{E}_{({\bf 3}; 2,1; \underline{2,1})}:\mathrm{DI}(\pi)| = |\widetilde{E}_{(3,2,1,-1,-2)}:\mathrm{DI}(\pi)|=
|E_{(1,1,1,0,0)}:\mathrm{DI}(\pi)| = {2 \choose 1 - 1} = 1,\\
&|\widetilde{E}_{(2,1; \underline{{\bf 3},2,1})}:\mathrm{DI}(\pi)| =
|\widetilde{E}_{(2,1,-1,-2,-3)}:\mathrm{DI}(\pi)| =
|E_{(0,0,-1,-1,-1)}:\mathrm{DI}(\pi)|= {2 \choose 1 - 0} = 2. 
\end{align*}
\hfill\qed
\end{example}

\begin{example} \label{eg-splitf4-parity}
Let us revisit Example 4.4 of \cite{DW4}. Adopt the setting there, and let $\varpi_1, \dots, \varpi_4$ be the fundamental weights for $\Delta^+(\frk, \frt)$. We have that
$$
\zeta=\varpi_4=(1, 1, 0, 0).
$$
We use $[a, b, c, d]$ to stand for the $\frk$-type $a\varpi_1+b\varpi_2+c\varpi_3+d\varpi_4$. Note that $w_0^K=-1$. For the $A_{\frq}(\lambda)$ module, we note that
$$
\Delta(\fru, \frt)\subseteq \Delta^+(\frk, \frt) \cup w^{(1)} \Delta^+(\frp, \frt),
$$
where $w^{(1)}=s_{\alpha_4}\in W(\frg, \frt)^1$. Then one can figure out that $\Delta(\fru\cap\frp, \frt)$ consists of the following eight roots:
\begin{align*}
(0, -1, 1, 0), (0, 1, 1, 0), (\frac{1}{2}, \frac{1}{2}, \frac{1}{2}, -\frac{1}{2}), (\frac{1}{2}, \frac{1}{2}, \frac{1}{2}, \frac{1}{2}), (1, 0, 0, -1), (1, 0, 0, 0), (1, 0, 0, 1), (1, 0, 1, 0).
\end{align*}
Here the first root $(0, -1, 1, 0)=-\alpha_4$, and it is the unique one in $\Delta^-(\frp, \frt)$.

One identifies that
$$
\lambda+2\rho(\fru\cap\frp)=[0, 0, 1, 3]
$$
is the lowest $K$-type of the $A_{\frq}(\lambda)$. As computed in Example 6.3 of \cite{DDY}, there is only one $\widetilde{K}$-type in $H_D(A_{\frq}(\lambda))$, i.e., $E_{\nu}$ with $\nu=[0, 0, 0, 1]$. It has multiplicity two. Indeed, there are two spin-lowest $K$-types of $A_{\frq}(\lambda)$ in total, both of which are related to the $\widetilde{K}$-type $E_{\nu}$:
$$
\mu_1:=[0, 2, 0, 4]=\lambda+2\rho(\fru\cap\frp)+(1, 0, 0, 1), \quad
\mu_2:=[0, 0, 3, 1]=\lambda+2\rho(\fru\cap\frp)+ 2 (-\alpha_4).
$$
Therefore, $\mu_1$ and $\mu_2$ have distinct parities. By \eqref{parity-delta-length-w}, it must happen that one $E_{\nu}$ lives in $H_D^+(A_{\frq}(\lambda))$, while the other $E_{\nu}$ lives in $H_D^-(A_{\frq}(\lambda))$. Thus they cancel in ${\rm DI} (A_{\frq}(\lambda))$, which then vanishes.
Indeed, as been explicitly obtained in Example 6.3 of \cite{DDY},
$$
w_1=w^{(2)}=s_4s_3, \quad w_2=w^{(10)}=s_4s_3s_2s_1s_3s_2s_4.
$$
Thus the $E_{\nu}$ from $E_{\mu_1}\otimes {\rm Spin}_G$ lives in $H_D^+(A_{\frq}(\lambda))$, while the $E_{\nu}$ from $E_{\mu_2}\otimes {\rm Spin}_G$ lives in $H_D^-(A_{\frq}(\lambda))$. \hfill\qed
\end{example}

\bigskip

\centerline{\scshape Funding}
Dong was supported by the National Natural Science Foundation of China (grant 11571097, 2016-2019). Wong is supported by the National Natural Science Foundation of China (grant 11901491) and the Presidential Fund of CUHK(SZ).

\medskip
\centerline{\scshape Acknowledgements}
We thank Professor Vogan sincerely for guiding us through coherent families.

\end{document}